\documentclass[11pt, reqno]{amsart}
\usepackage{amsthm,amsmath,amsfonts,amssymb,color}
\usepackage[bookmarks]{hyperref}
\reversemarginpar

\newtheorem{thm}{Theorem}[section]

\newtheorem{cor}[thm]{Corollary}
\newtheorem{lemma}[thm]{Lemma}

\newtheorem{preremark}[thm]{Remark}
\newenvironment{remark}{\begin{preremark}\rm}{\medskip \end{preremark}}

\numberwithin{equation}{section}

\newcommand{\abs}[1]{\left\vert#1\right\vert}

\newcommand{\R}{\mathbb R}

\DeclareMathOperator{\Vol}{Vol}

\newcommand{\dd} {\mathrm{d}}

\DeclareMathOperator{\dv}{div}
\DeclareMathOperator{\Def}{Def}
\DeclareMathOperator{\Ric}{Ric}

\def\H{\mathbb H^{2}(-a^{2})}
\def\be{\begin{equation}}
\def\ee{\end{equation}}

\newcommand{\ip}[1]{\langle{#1}\rangle}

\begin{document}
 \title[Asymptotic behavior on the hyperbolic plane]{Asymptotic behavior of the steady Navier-Stokes equation on the hyperbolic plane.}

\author[Chan]{Chi Hin Chan}
\address{Department of Applied Mathematics, National Chiao Tung University,1001 Ta Hsueh Road, Hsinchu, Taiwan 30010, ROC}
\email{cchan@math.nctu.edu.tw}

\author[Chen]{Che-Kai  Chen}
\address{Department of Applied Mathematics, National Chiao Tung University,1001 Ta Hsueh Road, Hsinchu, Taiwan 30010, ROC}
\email{ckchen7537\@gmail.com}

\author[Czubak]{Magdalena Czubak}
\address{Department of Mathematics\\
University of Colorado Boulder\\ Campus Box 395, Boulder, CO, 80309, USA}
\email{czubak@math.colorado.edu}

\begin{abstract}
We develop the asymptotic behavior for the solutions to the stationary Navier-Stokes equation in the exterior domain of the 2D hyperbolic space.  More precisely, given the finite Dirichlet norm of the velocity, we show the velocity decays to $0$ at infinity.  We also address the decay rate for the vorticity and the behavior of the pressure.
\end{abstract}
\maketitle

 \section{Introduction}
 Exterior domain is one of the fundamental domains studied in fluid mechanics.
The problem to be described has a satisfactory answer in three dimensions in the Euclidean setting, but there are questions that remain open in two dimensions, and they have been open since the work of Leray \cite{LerayExt}.  In this article, we show these questions can be answered if we pose them on the hyperbolic plane.  We begin by describing the problem and providing historical background.
 
Let $K$ be a compact set, an obstacle, in $\R^2$.  Consider a fluid surrounding $K$, where the behavior of the fluid is governed by the stationary Navier-Stokes equation.  Then the exterior domain problem in the $\R^2$ setting consists of finding a
smooth solution $u : \mathbb{R}^{2} - K \rightarrow \mathbb{R}^{2}$, and the pressure $p: \R^2-K \rightarrow \R$ satisfying
\begin{equation}\label{ext}
\begin{split}
-\triangle u + u\cdot \nabla u + \nabla p &= 0 ,\\
\dv u & = 0,\\
u|_{\partial K} &= 0,
\end{split}
\end{equation}
and $\int_{\mathbb{R}^{2} - K } |\nabla u|^{2} < \infty, \ \mbox{and such that}\quad u(x) \rightarrow \overline{u}_{\infty}\ \mbox{ as}\quad |x| \rightarrow \infty,$
where $\overline{u}_{\infty}\in \R^2$ is a given constant vector.  $\overline{u}_{\infty}$ represents the behavior of  the flow at the far range.  
 
The history of the problem and the settlement of the analogous problem in three dimensions begins with the work of Leray \cite{LerayExt}.  The method of Leray leads to a solution in three dimensions, but meets with a hurdle in 2D.  

The idea of Leray was to obtain a solution $(u_R, p_R$) in  $\{|x| \leq R\}\cap \mathbb{R}^{2} -K$ satisfying $u_{R}|_{\{|x| = R\}} = \overline{u}_{\infty}$ and the finite Dirichlet property in $\{|x| \leq R\}\cap \mathbb{R}^{2} -K$.  Then while the limiting solution, denoted by $u_{L}$, was shown to satisfy the finite Dirichlet norm in  $\mathbb{R}^{2} - K $,  the  behavior of $u_{L}$ at infinity was not known.  This was also an issue in 3D, but Finn \cite{Finn}, Ladyzhenskaya \cite{Lbook}, and Babenko \cite{Babenko} were able to bring the 3D problem to a positive conclusion.  The reason for this is that in 3D, the homogeneous $\dot H^1$ norm controls the $L^6$ norm of the difference $u_L-\overline u_\infty$ as well as  $\int_{\R^3-K}\frac{\abs{u_L-\overline u_\infty}^2}{r^2}$.   In 2D, the following holds
 \[
 \int_{\R^2-K}\frac{\abs{u_L-\overline u_\infty}^2}{r^2(\log r)^2}\leq C(1 + \int_{\R^2-K}\abs{\nabla u_L}^2).
 \]
 Unlike the 3D estimates, this estimate does not preclude $u_L$ from being trivial.  Essentially,  the failure of the energy method to produce good estimates in 2D is the source of the difficulty in completing the 2D problem.   
  
Important progress was made by Gilbarg and Weinberger \cite{GilbargWeinberger1974, GilbargWeinberger1978}, who in particular showed that a typical solution, 
$u : \mathbb{R}^{2}-K \rightarrow \mathbb{R}^{2}$ to \eqref{ext} with $\int_{\mathbb{R}^{2}-K} |\nabla u|^{2} < \infty$  (so not necessarily obtained by Leray's method)
satisfies the following
\begin{equation}\label{log}
\lim_{r\rightarrow \infty } \frac{|u(r,\theta )|^{2}}{\log r}  = 0.
\end{equation}
However, from \eqref{log} is not  clear if  $|u| \in L^{\infty }(\mathbb{R}^{2}-K )$.  
\eqref{log} is based on the finite Dirichlet norm of $u$ and a standard energy estimate.  On the other hand, they showed that the Leray solution $u_L$ has to be in $L^\infty$, and if $\overline u_\infty$ is trivial, then so is $u_L$ at infinity.

Subsequent breakthrough came from Amick \cite{Amick}, who indicated that one cannot improve \eqref{log} without taking into account the structure of the equation \eqref{ext}.  Amick was able
to prove  that the properties found by Gilbarg and Weinberger for Leray solutions hold for all solutions.  Moreover, he showed that the solution converges to  
\emph{some} nonzero vector ${u}_{\infty}$ in the far range for symmetric flows, and in certain sectors of the plane if the flow is not symmetric.
However, whether ${u}_{\infty}$ coincides with the prescribed $\overline{u}_{\infty}$, and if the pointwise convergence can be proved in general are questions that are still open.

In this paper, we answer these questions on the hyperbolic plane.\footnote{The 3D problem on the hyperbolic space will be considered in a forthcoming work by the second author.}  More precisely, let $a, R_0>0$, and consider  $$\Omega (R_0) = \mathbb{H}^2(-a^2) -\overline{ B_O(R_0)},$$ where $B_O(R_0)$ is a geodesic ball in a hyperbolic plane $\H$ with constant sectional curvature $-a^2$, and $O$ is a fixed base point in $\H$.  We study the following stationary Navier-Stokes equation on $\Omega (R_0)$, 
\begin{equation}\label{StatNSforvelocityDecay}
\begin{split}
2 \Def^* \Def v + \nabla_v v + \dd P & = 0 , \\
\dd^* v & = 0,
\end{split}
\end{equation}
where $P$ is a smooth function on $\Omega (R_0)$, and $2\Def^\ast\Def v=-2\dv \Def v,$ and $\Def$ is the deformation tensor, which can be written in coordinates as
\[
(\Def v)_{ij}=\frac 12(\nabla_i v_j+\nabla_j v_i).
\]
Moroever, a computation using Ricci identity shows for divergence free $v$ that on the hyperbolic plane
\[
2\Def^\ast\Def v=-\Delta v -2\Ric v=-\Delta v + 2a^2 v,
\]
where $-\Delta$ is the Hodge Laplacian.  We use this operator as we believe this is the correct form of the equations on a Riemannian manifold as indicated in \cite{EbinMarsden}.  For an extended discussion about the possible forms of the equations, see \cite{CCD16}.

We assume that just like on $\R^2$, $v$ satisfies  the finite Dirichlet property
\begin{equation}\label{FDIforVelDecay}
\int_{\Omega (R_0)} \big | \nabla v \big |_a^2 \Vol_{\mathbb{H}^2(-a^2)} < \infty .
\end{equation}
Without prescribing any conditions on the boundary of the obstacle, we show that $v$ must vanish at infinity.

\begin{thm}\label{VelocityDecayThm}
Let $R_0 > 0$, and suppose $v$ is a smooth $1-$form that solves \eqref{StatNSforvelocityDecay}  on $\Omega (R_0)$, and 
  satisfies the finite Dirichlet norm property   \eqref{FDIforVelDecay}.
Then, it follows that we have the following decay property of $v$ in the far range.
\begin{equation}\label{DecayofVel}
\lim_{\rho (x) \rightarrow \infty } \big \|v\big \|_{L^{\infty} (B_{x} (r(a)) )} = 0 ,
\end{equation}
where $\rho (x)$ is the geodesic distance of $x$ from the center $O$ of the obstacle $B_O(R_0)$ in $\mathbb{H}^2(-a^2)$ and  
\begin{equation*}
r(a) = \frac{1}{a} \log \Big( \frac{1+3e^a}{3+e^a } \Big ).
\end{equation*}
\end{thm}
(The reason for the form of $r(a)$ is explained in Section \ref{prelim}.) Then together with smoothness of $v$, we immediately get
\begin{cor}  Let $R_1>R_0$.
Suppose $v$ is a smooth $1-$form that solves \eqref{StatNSforvelocityDecay}  on $\Omega (R_0)$, and 
  satisfies  \eqref{FDIforVelDecay}.  Then $v\in L^\infty(\Omega(R_1)).$

\end{cor}
We also address the decay of the vorticity at infinity. 
\begin{thm}\label{FinalThmforExpdecay}
Let  $R_1>R_0$, and let $v$ be a smooth $1-$form that solves \eqref{StatNSforvelocityDecay} on $\Omega (R_0)$, which satisfies \eqref{FDIforVelDecay}. Let $\omega = * \dd v$  be the associated vorticity of $v$. We consider the positive constant
\begin{equation}\label{definedelta}
\delta\equiv\delta (a, \|v\|_{L^\infty(\Omega(R_1))}) = \frac{1}{4}  \Big \{  \Big ( ( \|v\|_{L^\infty(\Omega(R_1))} - a )^2 + 8a^2     \Big )^{\frac{1}{2}} - (\|v\|_{L^\infty(\Omega(R_1))} - a)  \Big \}.
\end{equation}
Then, the following apriori estimate holds for any $x \in \Omega (R_1)$.
\begin{equation}\label{Finalexpdecay}
-A e^{-\delta  \rho (x)} \leq \omega (x) \leq A e^{-\delta  \rho (x)} ,
\end{equation}
where 
\begin{equation}\label{definitionofA}
A = \exp ( \delta R_1  )   \big \| \omega \big \|_{L^{\infty} (\partial B_O(R_1))}.
\end{equation}
\end{thm}
 
Gilbarg and Weinberger use the vorticity equation and first establish decay rates for the vorticity, and then move on to showing $L^\infty$ bounds for $v$.  What we found is that in the hyperbolic setting, the $L^\infty$ bounds are easier to obtain due to better estimates than in the Euclidean 2D setting.  The key idea is the use of a Poincar\'e type inequality on an exterior domain to obtain a uniform control on the $L^2$ norm of the solution.  Such inequality on the whole hyperbolic space was established by the first and third author in \cite{CC15}.  To show it here, we follow the approach from \cite{CC15} combined with test functions used by Gilbarg and Weinberger \cite{GilbargWeinberger1978}.

Initial attempts to adapt the proof for the vorticity decay as in \cite{GilbargWeinberger1978} to the hyperbolic plane were not successful, so we ended up using a geometric approach inspired by the work of Anderson and Schoen \cite{AndersonSchoen}. There, Perron's method with barrier function $e^{-\delta \rho (x)}$ is applied to the Laplacian on a negatively curved manifold.  We apply that idea to an elliptic equation for the vorticity that can be  obtained by taking $* \dd$ on both sides of the first line of  \eqref{StatNSforvelocityDecay}.  The equation is 
\begin{equation}\label{vorticityeq}
-\Delta \omega + 2a^2 \omega + g(v , \nabla \omega) = 0,
\end{equation}
where $g$ is the metric on the hyperbolic plane.

So we consider the elliptic operator 
\begin{equation}\label{eoperator}
L (f) = \Delta f - 2a^2  f - g(v  ,\nabla f),
\end{equation}
and construct subsolutions and supersolutions $\pm A e^{-\delta \rho (x)}$.

Finally we show that the property of the pressure obtained by Gilbarg and Weinberger \cite{GilbargWeinberger1978} cannot be expected in general.
\begin{thm}\label{thm_p}
Let $R_0>0$.  There exist $(v, P)$ that satisfy  \eqref{StatNSforvelocityDecay} on $\Omega (R_0)$, are both smooth, and such that $v$ has finite Dirichlet property \eqref{FDIforVelDecay}, but there exist no constant $L$ such that 
\[
\lim_{\rho (x) \rightarrow \infty } |P(x)-L|=0.
\]
\end{thm}

\subsection{Organization of the paper}
 In Section \ref{prelim} we set up the Poincar\'e model for the hyperbolic plane, and introduce the function spaces that will be used throughout the paper.  Section \ref{section_Linfty} is devoted to showing the solution to the Stokes equation can be estimated locally in $L^\infty$.  The strategy here is to rely on the well-developed theory of a priori estimates in the Euclidean setting.  Therefore, we start with the intrinsic Stokes equation on the hyperbolic plane, and then we write it in terms of the Euclidean derivatives on the Poincar\'e disk (see equation \eqref{Equation3.14NEW}).  In Section \ref{section_v_decay} we derive the Poincar\' e type estimate on the exterior domain, and then apply it together with the result of Section \ref{section_Linfty} to prove Theorem  \ref{VelocityDecayThm}, the decay of the velocity at infinity.  The decay rate for the vorticity is obtained in Section \ref{section_vorticity}, and Section \ref{section_p} discusses the pressure.  In the appendix \ref{appendixa} we include what should be a standard material for the $L^\infty$ bound for the solution of the Stokes equation.
 
\subsection{Acknowledgments}
The first and third author would like to thank Vladim{\'i}r  \v{S}ver{\'a}k for introducing us to the problem of the exterior domain.  C. H. Chan is partially supported by a 
grant from the National Science Council of Taiwan (NSC 101-2115-M-009-016-MY2).
M. Czubak is partially supported by a grant from the Simons Foundation \# 246255, and would like to also thank MSRI, where part of this work was carried out.

\section{Preliminaries}\label{prelim}

\subsection{Hyperboloid model}\label{Hyperboloid}
The hyperboloid model for the hyperbolic space $\Bbb{H}^{2}(-a^{2})$ is given by
\begin{equation}
\mathbb{H}^2(-a^2) = \big\{ (x_0, x_1 , x_2 ) : x_0^2 - x_1^2 - x_2^2 = \frac{1}{a^2} , x_0 > 0 \big\}\subset \R^3.
\end{equation}
For each $x = (x_0, x_1, x_2) \in \mathbb{R}^3$, the tangent space $T_x\mathbb{R}^3$ can be equipped with the following symmetric quadratic form
\begin{equation}\label{lorentz}
\ip{v, w} = -v_0w_0 + v_1w_1 + v_2w_2 ,  \quad v,w \in T_x\mathbb{R}^3.
\end{equation}
Then the Riemannian metric $g(\cdot ,\cdot )$ on $\mathbb{H}^2(-a^2)$ is induced through the restriction of $\ip{\cdot ,\cdot}$ onto the tangent bundle of the submanifold $\H$.
In other words, for each point $x \in \mathbb{H}^2(-a^2)$, $g(\cdot , \cdot )_x$ is given by the following relation
\begin{equation}
g(\cdot , \cdot )_{x} = \ip{\cdot , \cdot }\big |_{x} .
\end{equation}

From now on, we write a point $x = (x_0, x_1, x_2)$ as $x = (x_0 , x')$, with $x' = (x_1 , x_2)$.

In general, the geodesic ball at $x$ with radius $R$ in $\Bbb H^{2}(-a^{2})$ will be denoted by $$B_{x}(R)=\{y\in \mathbb H^{2}(-a^{2}):\rho(x,y)< R\},$$
where $\rho(x,y)$ is the geodesic distance between $x$ and $y$ in $\Bbb{H}^{2}(-a^{2}).$ For any $x\in \Bbb R^{2}$ and $R>0$, the Euclidean open ball centered at $x$ with radius $R$ will be denoted by $$D_{x}(R)=\{y\in \Bbb R^{2}:|x-y|<R\}.$$

Next, we consider the unit disc $D_{0}(1)$ in $\Bbb R^{2}$ and the smooth mapping $Y: \mathbb H^{2}(-a^{2})\rightarrow D_{0}(1)$ defined by $$Y(x)=\frac{x'}{x_{0}+\frac{1}{a}},\quad x=(x_{0},x')\in \Bbb H^{2}(-a^{2}).$$
 The map $Y$ maps $\Bbb H^{2}(-a^{2})$ bijectively onto $D_{0}(1)$ with a smooth inverse, so $Y$ can be chosen as a coordinate system on the manifold $\Bbb H^{2}(-a^{2}).$

 The inverse map $Y^{-1}: D_{0}(1)\rightarrow \Bbb H^{2}(-a^{2})$ is given by $$Y^{-1}(y)=\bigg(\frac{2}{a(1-|y|^{2})}-\frac{1}{a},\frac{2y_{1}}{a(1-|y|^{2})},\frac{2y_{2}}{a(1-|y|^{2})}\bigg),\quad y=(y_1,y_2)\in D_{0}(1). $$ 
 
 Using $Y$ we can identify $\H$ with $D_0(1)$ equipped with the metric $\frac{4}{a^2(1-\abs{y}^2)^2}dy^i\otimes dy^i$. So this is the Poincar\'e disk model.  Now, let $\tilde y \in D_0(1)$ with $\abs{\tilde y}=r$, then by parametrizing the straight line connecting $0$ and $\tilde y$, we see that the geodesic distance between $0$ and $\tilde y$ is (see for example \cite{redbook})
 \be\label{gdist}
\rho(0,\tilde y)=\frac 1a\int^r_0 \frac{2}{1-t^2} d t =\frac 1a \log(\frac{1+r}{1-r}).
 \ee
So if we would like to talk about a geodesic ball $B_O(R)\subset \H$, and relate it to a Euclidean ball in the unit disk, then we need to find $r$ such that
\[
\frac 1a \log(\frac{1+r}{1-r})=R.
\]
A computation shows that $$r=\tanh(\frac{a}{2}R),$$ so $Y$ maps a geodesic ball of radius $R$ onto the Euclidean ball of radius $\tanh(\frac{a}{2}R)$, i.e., $Y\big(B_{O}(R)\big)=D_{0}\big(\tanh(\frac{a}{2}R)\big).$  The way this is employed is that we will start with a ball of radius $1$ on the hyperbolic plane, so that means doing estimates on the Euclidean ball of radius $\tanh(\frac{a}{2})$.  Then at some point we go from the estimates on the ball of radius $\tanh(\frac{a}{2})$ to $\frac 12$ of $\tanh(\frac{a}{2})$ (e.g. when applying 
\eqref{BootstrapFinalNEW}), so when we go back to the hyperbolic plane, this maps to a ball of radius
\[
\frac 1a \log \Big(\frac{1+\frac 12\tanh(\frac{a}{2})}{1-\frac 12\tanh(\frac{a}{2})}\Big)= \frac{1}{a} \log \Big ( \frac{1+3 e^a}{3+e^a} \Big ).
\]
This explains the reason for the choice of $r(a)$ in Theorem \ref{VelocityDecayThm}.

We now introduce several function spaces, which will be used in this article.
\subsection{Function spaces}
Let $M$ be a Riemannian manifold with a Riemannian mteric $g_{_{M}}$, and let $\nabla^{M}$ be the Levi-Civita connection on $M$. Consider a domain $\Omega$ in $M$. We define the following function spaces:
\begin{itemize}
\item $\bigwedge^{k}(\Omega)$ is the space of all smooth $k$-forms in $\Omega$.

\item $\bigwedge^{k}_{c}(\Omega)$ is the space of all smooth $k$-forms with compact support in $\Omega$.
 
\item$\bigwedge^{k}_{\sigma}(\Omega)$ is the space of all smooth, $\dd^{*}$-closed, $k$-forms on $\Omega.$
 
\item $L^{k,p}(\Omega)$ is the space of all weakly differentiable $1$-forms $v$ with $(\nabla^M)^{k}v\in L^{p}(\Omega)$. $L^{k,p}(\Omega)$ is equipped with the  semi-norm $\|v\|_{_{L^{k, p}(\Omega)}}=\|(\nabla^M)^{k}v\|_{_{L^{p}(\Omega)}},$ and $L^{k,p}_{0}(\Omega)$ is the closure of $\bigwedge^1_c(\Omega)$ in $L^{k,p}(\Omega).$
\item $W^{k,p}(\Omega)$ is the Sobolev space which consists of all weakly differentiable $1$-forms $v$ with $(\nabla^M)^{\alpha}v\in L^{p}(\Omega)$ for all $0 \leq \alpha\leq k$. $W^{k,p}(\Omega)$ is equipped with the norm $\|v\|_{_{W^{k, p}(\Omega)}}=\sum_{\alpha=0}^{k}\|(\nabla^M)^{\alpha}v\|_{_{L^{p}(\Omega)}},$ and $W^{k,p}_{0}(\Omega)$ is the closure of $\bigwedge^1_c(\Omega)$ in $W^{k,p}(\Omega).$
 \end{itemize}
For the case of $p=2$, we write $W^{k,2}(\Omega)=H^{k}(\Omega)$, $W^{k,2}_{0}(\Omega)=H^{k}_{0}(\Omega).$
 
In order to simplify our notation, the Levi-Civita connection $\nabla^{\Bbb H^{2}(-a^{2})}$ on the hyperbolic space $\mathbb{H}^2(-a^2)$ will be denoted by $\nabla$. We use $C_0$ to denote an absolute constant in each inequality estimate which could change from line to line.

\section{Local $L^\infty$ bound on $v$}\label{section_Linfty}
The purpose of this section is to show we can obtain a bound on $L^\infty$ norm of $v$ on a small enough ball in the hyperbolic plane, where $v$ is a solution to the Stokes equation.   First we consider a general $u$, not necessarily a solution to the Stokes equation, and prove a bound on the Dirichlet norm of the pull-back of $u$ to the Poincar\'e disk.  The bound is in terms of the intrinsic $L^2$ and Dirichlet norms.

\begin{lemma}\label{Lemma3.1NEW}
The following estimate holds for any $1-$form $u \in H^1(B_{O}(1))$, where $u^{\sharp}$ is the pull back of $u$ via the map $Y^{-1}$.
\begin{equation}\label{Equation3.1NEW}
\big \|\nabla^{\mathbb{R}^2} u^{\sharp} \big \|_{L^2(D_0(\tanh (\frac{a}{2})))}^2 \leq  32\Big  \{ \frac{1}{a^2} \cosh^4\Big ( \frac{a}{2}\Big ) \big \|\nabla u \big \|_{L^2(B_O(1))}^2 +\sinh^2a  \big \| u \big \|_{L^2(B_O(1))}^2 \Big \} .
\end{equation}
\end{lemma}

\begin{proof}

Now, for any $1-$form $u$ on $(B_{O}(1))$, the pull back of $u$ is given by $$ u^{\sharp}:=(Y^{-1})^{*}u=(u_{1}\circ Y^{-1})dy^{1}+(u_{2}\circ Y^{-1})dy^{2}.$$ Write $u_{\alpha}^{\sharp}=u_{\alpha}\circ Y^{-1}$ for $\alpha=1, 2,$ and let $\nabla$ be the induced Levi-Civita connection acting on smooth $1$-forms on $\Bbb{H}^{2}(-a^{2})$.  Then  (see \cite[Appendix]{CC13})

\begin{equation}\label{identity3.2NEW}
\begin{split}
\nabla u= &\bigg\{\frac{\partial u_{1}}{\partial Y^{1}}-\frac{2Y^{1}u_{1}}{1-|Y|^{2}}+\frac{2Y^{2}u_{2}}{1-|Y|^{2}}\bigg\}dY^{1}\otimes dY^{1}\\
&\quad+\bigg\{\frac{\partial u_{2}}{\partial Y^{1}}-\frac{2Y^{2}u_{1}}{1-|Y|^{2}}-\frac{2Y^{1}u_{2}}{1-|Y|^{2}}\bigg\}dY^{1}\otimes dY^{2}\\
&\quad+\bigg\{\frac{\partial u_{1}}{\partial Y^{2}}-\frac{2Y^{2}u_{1}}{1-|Y|^{2}}-\frac{2Y^{1}u_{2}}{1-|Y|^{2}}\bigg\}dY^{2}\otimes dY^{1}\\
&\quad+\bigg\{\frac{\partial u_{2}}{\partial Y^{2}}+\frac{2Y^{1}u_{1}}{1-|Y|^{2}}-\frac{2Y^{2}u_{2}}{1-|Y|^{2}}\bigg\}dY^{2}\otimes dY^{2}.\\
\end{split}
\end{equation}
We consider the orthonormal 
 frame $\{e_1^*, e_2^* \}$ of $T^*(\mathbb{H}^2(-a^2))$ given by 
\begin{equation}\label{Dualframe}
e_j^* = \frac{2}{a(1-|Y|^2)} \dd Y^j, \quad j=1, 2.
\end{equation}  
Hence $\{e_i^* \otimes e_j^* : 1\leq i,j\leq 2\}$ constitutes an orthonormal frame on $T^*(\mathbb{H}^2(-a^2))\otimes T^*(\mathbb{H}^2(-a^2))$, and it follows that
\be\label{normdY}
\abs{dY^j\otimes dY^k}=\frac{a^2(1-\abs{Y}^2)^2}{4}\delta^{jk}.
\ee

To obtain \eqref{Equation3.1NEW}, we have to estimate the absolute value of the partial derivatives of $u_{\beta}^{\sharp}$ with respect to $y^{\alpha}$ for all $\alpha$ and $\beta$ equal $1$ or $2$. We just estimate $|\partial_{y_{1}}u_{1}^{\sharp}|$ to illustrate the idea, then the estimates for all other terms follow basically in the same manner.

First, we observe that by \eqref{normdY}
 $$|\nabla u|_{a}\geq \frac{a^2(1-|Y|^{2})^2}{4}\Big |\frac{\partial u_{1}}{\partial Y^{1}}-\frac{2Y^{1}u_{1}}{1-|Y|^{2}}+\frac{2Y^{2}u_{2}}{1-|Y|^{2}}\Big|.$$

 Thus, by the triangle inequality $$\Big|\frac{\partial u_{1}}{\partial Y^{1}}\Big|\leq \Big |\frac{\partial u_{1}}{\partial Y^{1}}-\frac{2Y^{1}u_{1}}{1-|Y|^{2}}+\frac{2Y^{2}u_{2}}{1-|Y|^{2}}\Big|+\frac{4|Y||u|}{1-|Y|^{2}}\leq
  \frac{4}{a^2(1-|Y|^{2})^2}|\nabla u|_{a}+\frac{4|Y||u|}{1-|Y|^{2}}.$$
These imply the following pointwise estimate on $D_{0}(\tanh(\frac{a}{2}))$,
\begin{equation}\label{PointwiseestimateNEW}
\Big|\frac{\partial u^{\sharp}_{1}}{\partial y^{1}}\Big|\leq \frac{4}{a^2(1-|y|^{2})^2}|\nabla u|_{a}\circ Y^{-1}+\frac{4|y||u^{\sharp}|}{1-|y|^{2}}.
\end{equation}

Next, using $(a+b)^2\leq 2a^2+2b^2$, $\cosh^2\theta-\sinh^2\theta=1$ and the definition of the integration on manifolds
\begin{align}
\int_{D_{0}(\tanh (\frac{a}{2}))}\Big|\frac{\partial u^{\sharp}_{1}}{\partial y^{1}}\Big|^{2}dy^{1}\wedge dy^{2}&\leq 2\int_{D_{0}(\tanh (\frac{a}{2}))} \Big(\frac{4}{a^2(1-|y|^{2})^2}\Big)^2|\nabla u  |^2_{a} \circ Y^{-1}dy^{1}\wedge dy^{2}\nonumber\\
&\qquad+2\int_{D_{0}(\tanh (\frac{a}{2}))} \frac{16|y|^2|u^\sharp|^2}{(1-|y|^{2})^2}dy^{1}\wedge dy^{2}\nonumber\\
&\leq \frac{8}{a^2}\cosh^4(\frac a2)\int_{D_{0}(\tanh (\frac{a}{2}))} \frac{4}{a^2(1-|y|^{2})^2}|\nabla u  |^2_{a} \circ Y^{-1}dy^{1}\wedge dy^{2}\nonumber\\
&\qquad+32\int_{D_{0}(\tanh (\frac{a}{2}))} \frac{|y|^2|u^\sharp|^2}{(1-|y|^{2})^2}dy^{1}\wedge dy^{2}\nonumber\\
&\leq\quad  \frac{8}{a^2}\cosh^4(\frac a2)\int_{B_O(1)} |\nabla u  |^2_{a}  \frac{4}{a^2(1-|Y|^{2})^2} dY^{1}\wedge dY^{2}\nonumber\\
&\qquad +32\tanh^{2}\Big(\frac{a}{2}\Big)\cosh^{4}\Big(\frac{a}{2}\Big)\int_{D_{0}(\tanh (\frac{a}{2}))} |u^\sharp|^2 dy^{1}\wedge dy^{2}\nonumber\\
&=\quad \frac{8}{a^2}\cosh^4(\frac a2)\int_{B_O(1)} |\nabla u  |^2_{a} \Vol_{\H}\nonumber\\
&\qquad + 8\sinh^2(a)\int_{D_{0}(\tanh (\frac{a}{2}))}|u^{\sharp}|^{2}dy^{1}\wedge dy^{2} \nonumber
\end{align}

The above estimate still works if $\frac{\partial u^{\sharp}_{1}}{\partial y^{1}}$ is replaced by $\frac{\partial u^{\sharp}_{i}}{\partial y^{j}}$ for any $1 \leq i,j \leq 2$. 

Hence \eqref{Equation3.1NEW} follows.
\end{proof}

We are now ready to consider the Stokes equation.

\begin{lemma}\label{Supnormelliptic}
Consider a smooth $1$-form $v \in \Lambda^1_{\sigma}(B_O(1))$ and a smooth function $P \in C^{\infty} (B_O(1))$ which satisfy the following Stokes equation on $B_O(1)$
\begin{equation}\label{LinearStokeswithForcingNEW}
\begin{split}
2\Def^* \Def v + \dd P & = F , \\
\dd^* v & = 0 ,
\end{split}
\end{equation}
where $F \in \Lambda^1(B_O(1)) \cap L^{\frac{4}{3}}(B_O(1))$. Let
\begin{equation}\label{quitesimpler}
r(a) = \frac{1}{a} \log \Big ( \frac{1+3 e^a}{3+e^a} \Big ) ,
\end{equation}
where $r(a)$ is such that $Y(B_O(r(a)))=D_0(\frac{1}{2} \tanh \big(\frac{a}{2}\big))$ (using \eqref{gdist}). Then, it follows that $v$ satisfies the following a priori estimate.
\begin{equation}\label{Supnormestimate}
\big \| v \big \|_{L^{\infty} (B_O(r(a)))} \leq C_0 \Big \{ A_1(a)  \big \| F \big \|_{L^{\frac{4}{3}} (B_O(1)) } + A_2(a) \big \|v  \big \|_{L^2(B_O(1))} + A_3(a) \big \| \nabla v \big \|_{L^2(B_O(1))}  \Big \} ,
\end{equation}
where $C_0 > 0$ is an absolute constant which is independent of $a$, and where the constants $A_1(a)$, $A_2(a)$, $A_3(a)$ can be given explicitly as follows.
\begin{equation}\label{aboutA}
\begin{split}
A_1(a) & = a^{-\frac{1}{2}} \Big ( \tanh \Big(\frac{a}{2}\Big) \Big )^{\frac{1}{2}} \cosh^4\Big(\frac{a}{2}\Big) \cosh a, \\
A_2(a) & = a\bigg \{ \tanh \Big(\frac{a}{2}\Big) \Big ( \cosh^4 \Big(\frac{a}{2}\Big) \cosh^2 a + \sinh^2 a \Big )  +  \coth \Big(\frac{a}{2}\Big) + \sinh a \bigg \} ,\\
A_3(a) & =\cosh^2\Big(\frac{a}{2}\Big) \bigg \{ \tanh \Big(\frac{a}{2}\Big) \sinh a  \cdot\Big ( \cosh^2\Big(\frac{a}{2}\Big) \cosh a + 1  \Big ) + 1 \bigg \}.
\end{split}
\end{equation}
\end{lemma}

\begin{proof}
Consider a $1$-form $v \in \Lambda^1(B_O(1))$ and a smooth function $P \in C^{\infty}(B_O(1)),$ which satisfy equation \eqref{LinearStokeswithForcingNEW}. Under the coordinate system $Y : \mathbb{H}^2(-a^2) \rightarrow D_0(1)$, we express $v$ as $v = v_1 \dd Y^1 + v_2 \dd Y^2$. We also express $F$ as $F = F_1 \dd Y^1 + F_2 \dd Y^2$.
For each $j =1,2$, we define the function $v_j^{\sharp}$ by $v_j^{\sharp} = v_j \circ Y^{-1}$, and the function $F_j^{\sharp}$ by $F_j^{\sharp} = F_j \circ Y^{-1}$. We also write $P^{\sharp} = P \circ Y^{-1}$. Then, saying that the pair $(v, P)$ satisfies \eqref{LinearStokeswithForcingNEW} on the geodesic ball $B_O(1)$ is equivalent to saying that the $\mathbb{R}^2$-valued function $v^{\sharp} = (v_1^{\sharp} , v_2^{\sharp} )$ and the function $P^{\sharp}$ satisfy the following system of equations on the Euclidean disc $D_0\big(\tanh \big(\frac{a}{2}\big)\big)$ (see \cite{CC13}).
\begin{equation}\label{Equation3.14NEW}
\begin{split}
   \frac{a^2 (1-|y|^2)^2}{4}\Big(-\Delta^{\Bbb R^{2}}v^{\sharp}_{1}+\frac{4y^2(\partial_{2}v^{\sharp}_{1}-\partial_{1}v^{\sharp}_{2})}{1-|y|^2} \Big) + 2a^2 v_1^{\sharp} +\partial_1 P^\sharp &= F_1^{\sharp} ,\\
  \frac{a^2 (1-|y|^2)^2}{4}\Big(-\Delta^{\Bbb R^{2}}v^{\sharp}_{2}+\frac{4y^1(\partial_{2}v^{\sharp}_{1}-\partial_{1}v^{\sharp}_{2})}{1-|y|^2} \Big ) + 2a^2 v_2^{\sharp}+\partial_2 P^\sharp&=F_2^{\sharp},\\
 \text{div}\ v^{\sharp}&=0 .
\end{split}
\end{equation}
By a direct computation, we get
\begin{equation}
\begin{split}
\int_{B_O(1)} \big | F \big |_a^{\frac{4}{3}} \Vol_{\mathbb{H}^2(-a^2)} & = \int_{D_0(\tanh (\frac{a}{2}))} \big ( (F_1^{\sharp})^2 + (F_2^{\sharp})^2 \big )^{\frac{2}{3}} \bigg ( \frac{4}{a^2 (1-|y|^2)^2} \bigg )^{\frac{1}{3}} \dd y^1 \wedge \dd y^2 \\
& \geq \frac{4^{\frac{1}{3}}}{a^{\frac{2}{3}}} \int_{D_0(\tanh (\frac{a}{2}))} \big | F^{\sharp}\big |^{\frac{4}{3}} \Vol_{\mathbb{R}^2},
\end{split}
\end{equation}
which immediately gives 
\begin{equation}\label{FL-12NEW}
\big \| F^{\sharp} \big \|_{L^{\frac{4}{3}}(D_0(\tanh (\frac{a}{2})))} \leq \Big(\frac{a}{2}\Big)^{\frac{1}{2}}
\big \| F\big \|_{L^{\frac{4}{3}}(B_O(1))} .
\end{equation}

Next, for convenience we rephrase \eqref{Equation3.14NEW} as 
\begin{equation}\label{Equation1NEW}
\frac{a^2 (1-|y|^2)^2}{4}\big(-\Delta^{\Bbb R^{2}}v^{\sharp} \big ) +a^2 (1-|y|^2) (\partial_2v_1^{\sharp} - \partial_1 v_2^{\sharp} ) y^{\perp} + 2a^2 v^{\sharp} + \nabla P^\sharp = F^{\sharp},
\end{equation}
where $y^{\perp} = (y^2 , y^1)$.  

Next, we have to estimate $\big \|\nabla P \big \|_{L^{-1,2}(D_0(\tanh (\frac{a}{2})))}$.  To that end, we first estimate  
\begin{equation*}
 \bigg \|  \frac{a^2 (1-| \cdot |^2)^2}{4}  \big ( -\Delta^{\Bbb R^{2}}v^{\sharp} \big ) \bigg \|_{L^{-1,2}(D_0(\tanh (\frac{a}{2})))}.
\end{equation*}

Let $\varphi \in C^\infty_{c} (D_0(\tanh (\frac{a}{2})))$, then
\begin{equation*}
\begin{split}
& \bigg | \bigg <  \frac{a^2 (1-| \cdot |^2)^2}{4}   \big ( -\Delta^{\Bbb R^{2}}v^{\sharp} \big ) ,  \varphi    \bigg > _{L^{-1,2}(D_0(\tanh (\frac{a}{2}))) \otimes L^{1,2}_0(D_0(\tanh (\frac{a}{2})))} \bigg | \\
= & \bigg | \int_{D_0(\tanh (\frac{a}{2}))} \frac{a^2 (1-| y |^2)^2}{4}   \big ( -\Delta^{\Bbb R^{2}}v^{\sharp} \big ) \cdot  \varphi      \bigg | \\
= & \bigg | \int_{D_0(\tanh (\frac{a}{2}))} \nabla^{\Bbb R^{2}} v^{\sharp} :  \nabla^{\Bbb R^{2}} \bigg \{ \frac{a^2 (1-|y|^2)^2}{4} \cdot \varphi    \bigg \}  \bigg | \\
\leq & \bigg |  \int_{D_0(\tanh (\frac{a}{2}))} \nabla^{\Bbb R^{2}} v^{\sharp} : \nabla^{\Bbb R^{2}} \bigg ( \frac{a^2 (1-|y|^2)^2}{4}\bigg ) \cdot \varphi     \bigg | + \bigg | \int_{D_0(\tanh (\frac{a}{2}))}\frac{a^2 (1-|y|^2)^2}{4} \nabla^{\Bbb R^{2}} v^{\sharp} : \nabla^{\Bbb R^{2}} \varphi  \bigg | \\
\leq & \bigg | \int_{D_0(\tanh (\frac{a}{2}))}   a^2  \big (1-|  y |^2 \big ) \nabla^{\Bbb R^{2}} v^{\sharp} : y \varphi  \bigg | + \frac{a^2}{4}
\big \| \nabla^{\Bbb R^{2}} v^{\sharp} \big \|_{L^2(D_0(\tanh (\frac{a}{2})))}\cdot \big \| \nabla^{\Bbb R^{2}} \varphi  \big \|_{L^2(D_0(\tanh (\frac{a}{2})))} \\
\leq & a^2 \tanh \Big (\frac{a}{2}\Big ) \int_{D_0(\tanh (\frac{a}{2}))} \big | \nabla^{\Bbb R^{2}} v^{\sharp}\big |  \big | \varphi \big |  + \frac{a^2}{4}
\big \| \nabla^{\Bbb R^{2}} v^{\sharp} \big \|_{L^2(D_0(\tanh (\frac{a}{2})))} \big \| \nabla^{\Bbb R^{2}} \varphi  \big \|_{L^2(D_0(\tanh (\frac{a}{2})))} \\
\leq & a^2 \tanh \Big (\frac{a}{2}\Big ) \big \| \nabla^{\Bbb R^{2}} v^{\sharp} \big \|_{L^2(D_0(\tanh (\frac{a}{2})))} \big \| \varphi \big \|_{L^2(D_0(\tanh (\frac{a}{2})))} + \frac{a^2}{4}
\big \| \nabla^{\Bbb R^{2}} v^{\sharp} \big \|_{L^2(D_0(\tanh (\frac{a}{2})))} \big \| \nabla^{\Bbb R^{2}} \varphi  \big \|_{L^2(D_0(\tanh (\frac{a}{2})))} \\
\leq & C_0 a^{2} \Big ( 1+ \tanh^2 \Big (\frac{a}{2}\Big )\Big ) \big\| \nabla^{\Bbb R^{2}} v^{\sharp} \big \|_{L^2(D_0(\tanh (\frac{a}{2})))}  \big\| \nabla^{\Bbb R^{2}} \varphi  \big \|_{L^2(D_0(\tanh (\frac{a}{2})))} .
\end{split}
\end{equation*}
In the last line of the above estimate, we  employed the standard Poincar\'e inequality

\begin{equation}\label{PoincareinequalityNEW}
\big\| \varphi \big\|_{L^2(D_0(r))} \leq C_0 r  \big\| \nabla^{\mathbb{R}^2} \varphi \big\|_{L^2(D_0(r))}.
\end{equation}
To summarize, we have
\begin{equation}\label{Estimate1NEW}
 \bigg \|  \frac{a^2 (1-| \cdot |^2)^2}{4}  \big ( -\Delta^{\Bbb R^{2}}v^{\sharp} \big ) \bigg \|_{L^{-1,2}(D_0(\tanh (\frac{a}{2})))} \leq  C_0A_0(a) \big\| \nabla^{\Bbb R^{2}} v^{\sharp} \big \|_{L^2(D_0(\tanh (\frac{a}{2})))} ,
\end{equation}
where the absolute constant $C_0$ is independent of $a$, and 
\[
A_0(a)= a^{2} \Big ( 1 + \tanh^2 \Big(\frac{a}{2}\Big) \Big ) .
\]
Also, it follows from  \eqref{FL-12NEW} that
the following estimate holds for any  $\varphi \in C^\infty_{c} \big(D_0 \big(\tanh \big(\frac{a}{2}\big)\big)\big) .$
\begin{equation*}
\begin{split}
\bigg | \int_{D_0(\tanh (\frac{a}{2}))} F^{\sharp} \cdot \varphi \bigg | & \leq \big \| F^{\sharp} \big \|_{L^{\frac{4}{3}} (D_0(\tanh (\frac{a}{2})))}  \big \| \varphi \big \|_{L^4 (D_0 (\tanh (\frac{a}{2} )))} \\
& \leq \big \| F^{\sharp} \big \|_{L^{\frac{4}{3}} (D_0(\tanh (\frac{a}{2})))} C_0 \big \| \varphi \big \|_{L^2 (D_0 (\tanh (\frac{a}{2} )))}^{\frac{1}{2}}  \big \| \nabla^{\mathbb{R}^2} \varphi \big \|_{L^2 (D_0 (\tanh (\frac{a}{2} )))}^{\frac{1}{2}} \\
& \leq C_0 \Big ( \tanh \Big(\frac{a}{2}\Big) \Big )^{\frac{1}{2}} \big \| F^{\sharp} \big \|_{L^{\frac{4}{3}} (D_0(\tanh (\frac{a}{2})))} \big \| \nabla^{\mathbb{R}^2} \varphi \big \|_{L^2 (D_0 (\tanh (\frac{a}{2} )))} \\
& \leq C_0 \Big (\frac{a}{2} \Big )^{\frac{1}{2}} \Big ( \tanh \Big(\frac{a}{2}\Big) \Big )^{\frac{1}{2}} \|F\|_{L^{\frac{4}{3}} (B_O(1))}  \big \| \nabla^{\mathbb{R}^2} \varphi \big \|_{L^2 (D_0 (\tanh (\frac{a}{2} )))} ,
\end{split}
\end{equation*}
which gives
\begin{equation}\label{L-12estimate2NEW}
\big \| F^{\sharp} \big \|_{L^{-1,2} (D_0 (\tanh (\frac{a}{2})))} \leq  C_0\Big (\frac{a}{2} \Big )^{\frac{1}{2}} \Big ( \tanh\Big (\frac{a}{2}\Big) \Big )^{\frac{1}{2}} \big \|F \big \|_{L^{\frac{4}{3}} (B_O(1))} .
\end{equation}

Next, estimate
\begin{align}
 & \bigg | \int_{D_0(\tanh (\frac{a}{2}))} a^2 (1-|y|^2)  (\partial_2 v_1^{\sharp} - \partial_1 v_2^{\sharp} )  y^{\perp} \cdot \varphi  \bigg | \nonumber\\
&\quad\leq  2a^2 \int_{D_0 (\tanh (\frac{a}{2}))} \big | \nabla^{\mathbb{R}^2} v^{\sharp} \big |  |y| |\varphi| \nonumber\\
&\quad\leq  2a^2 \tanh \Big(\frac{a}{2}\Big)  \int_{D_0(\tanh (\frac{a}{2}))} \big | \nabla^{\mathbb{R}^2} v^{\sharp} \big |  |\varphi| \nonumber\\
&\quad\leq  2a^2 \tanh \Big(\frac{a}{2}\Big)  \big \| \nabla^{\mathbb{R}^2} v^{\sharp}  \big \|_{L^2 (D_0 (\tanh (\frac{a}{2})))} \big \| \varphi \big \|_{L^2(D_0 (\tanh (\frac{a}{2})))}\nonumber \\
&\quad\leq  C_0  a^2  \tanh^2 \Big(\frac{a}{2}\Big)  \big \| \nabla^{\mathbb{R}^2} v^{\sharp}  \big \|_{L^2 (D_0 (\tanh (\frac{a}{2})))}  \big \| \nabla^{\mathbb{R}^2} \varphi \big \|_{L^2(D_0 (\tanh (\frac{a}{2})))} \label{L-12ThreeNEW},
 \end{align}
where in the last line we again used \eqref{PoincareinequalityNEW}.
 
Using the easy fact that $ \big \| v^{\sharp} \big \|_{L^2(D_0(\tanh (\frac{a}{2})))} =  \big \| v \big \|_{L^2(B_O(1))}$, we also have
\begin{equation*}
\begin{split}
\bigg | 2a^2 \int_{D_0(\tanh (\frac{a}{2}))} v^{\sharp}  \varphi \bigg | & \leq
2a^2 \big \| v^{\sharp} \big \|_{L^2(D_0(\tanh (\frac{a}{2})))} \big \| \varphi \big \|_{L^2(D_0(\tanh (\frac{a}{2})))} \\
& \leq C_0 a^2 \tanh \Big(\frac{a}{2}\Big)  \big \| v^{\sharp}  \big \|_{L^2(D_0(\tanh (\frac{a}{2})))} \big \| \nabla^{\mathbb{R}^2} \varphi \big \|_{L^2(D_0(\tanh (\frac{a}{2})))} \\
& = C_0 a^2 \tanh \Big(\frac{a}{2}\Big)  \big \| v \big \|_{L^2(B_O(1))}  \big \| \nabla^{\mathbb{R}^2} \varphi \big \|_{L^2(D_0(\tanh (\frac{a}{2})))}.
\end{split}
\end{equation*}
So
\begin{equation}\label{TrivialL-12NEW}
\big \| 2a^2 v^{\sharp} \big \|_{L^{-1,2} (D_0(\tanh (\frac{a}{2})))} \leq  C_0 a^2 \tanh \Big(\frac{a}{2}\Big)  \big \| v \big \|_{L^2(B_O(1))}.
\end{equation}

By combining estimates \eqref{Estimate1NEW}, \eqref{L-12estimate2NEW}, \eqref{L-12ThreeNEW}, and \eqref{TrivialL-12NEW}, we deduce  

\begin{equation}\label{L-12pressure}
\begin{split}
\big \| \nabla^{\mathbb{R}^2} P^\sharp \big \|_{L^{-1,2} (D_0(\tanh (\frac{a}{2})))}
\leq & C_0 a^2 \Big ( 1 + \tanh^2 \Big(\frac{a}{2}\Big) \Big )  \big \| \nabla^{\mathbb{R}^2} v^{\sharp} \big \|_{L^2 (D_0(\tanh (\frac{a}{2})))} \\
& + C_0 a^2 \tanh\Big (\frac{a}{2}\Big)   \big \| v \big \|_{L^2(B_O(1))}\\
& + C_0    a^{\frac{1}{2}}  \Big ( \tanh\Big(\frac{a}{2}\Big) \Big )^{\frac{1}{2}} \big \| F \big \|_{L^{\frac{4}{3}}(B_O(1))} .\\
\end{split}
\end{equation}

Now, by combining \eqref{Equation3.1NEW} and \eqref{L-12pressure}, we obtain
\begin{equation}\label{EstimatenablapressureNEW}
\begin{split}
\big \| \nabla^{\mathbb{R}^2} P^{\sharp} \big \|_{L^{-1,2} (D_0(\tanh (\frac{a}{2})))}
\leq & C_0  a^2 \bigg \{ \Big (  1 + \tanh^2 \Big(\frac{a}{2}\Big) \Big)  \sinh a + \tanh \Big(\frac{a}{2}\Big) \bigg \}   \big \| v \big \|_{L^2(B_O(1))} \\
& + C_0  a    \Big( \cosh^2\Big ( \frac{a}{2} \Big) + \sinh^2\Big (\frac{a}{2}\Big) \Big )  \big \| \nabla v \big \|_{L^2(B_O(1))} \\
& + C_0 a^{\frac{1}{2}}  \Big ( \tanh \Big(\frac{a}{2}\Big) \Big)^{\frac{1}{2}}  \|F \|_{L^{\frac{4}{3}} (B_O(1))} .
\end{split}
\end{equation}

At this point, we employ the following fact from the regularity theory for Navier-Stokes equation \cite{Seregin}.

\begin{itemize}
\item For any $R > 0$, and any $P \in L^1_{loc} ( D_0(R) )$ which satisfies $\nabla^{\mathbb{R}^2} P \in L^{-1,2} (D_0(R))$, it follows that there exists some $c \in \mathbb{R}$ such that
    \begin{equation}
    \big \| P - c\big \|_{L^2 (D_0(R))} \leq C_0 \big \| \nabla^{\mathbb{R}^2} P \big \|_{L^{-1,2} (D_0(R))} ,
    \end{equation}
    where the absolute constant $C_0 > 0$ is \emph{independent} of $R$. (We note the similarity with \eqref{EASYPressureNEW}.  We just want to stress the independence of $C_0$ from $R$).
\end{itemize}
So, it follows from \eqref{EstimatenablapressureNEW} that, we can find some $c \in \mathbb{R}$ such that $P^{\sharp}-c$ satisfies 

\begin{equation}\label{EstimatePminusLNEW}
\begin{split}
\big \|  P^{\sharp} - c \big \|_{L^2 (D_0(\tanh (\frac{a}{2})))}
\leq & C_0  a^2 \bigg \{\Big(  1 + \tanh^2 \Big(\frac{a}{2}\Big)\Big)  \sinh a + \tanh\Big (\frac{a}{2}\Big) \bigg \}   \big \| v \big \|_{L^2(B_O(1))} \\
& + C_0  a    \Big ( \cosh^2\Big( \frac{a}{2} \Big ) + \sinh^2 \Big(\frac{a}{2}\Big) \Big )  \big \| \nabla v \big \|_{L^2(B_O(1))} \\
& + C_0  a^{\frac{1}{2}}  \Big( \tanh \Big(\frac{a}{2}\Big) \Big)^{\frac{1}{2}}  \|F \|_{L^{\frac{4}{3}} (B_O(1))}.
\end{split}
\end{equation}
 
 Next, rearranging \eqref{Equation1NEW} we get

\begin{equation*}
\begin{split}
-\Delta^{\mathbb{R}^2} v^{\sharp} + \nabla^{\mathbb{R}^2} \bigg \{\frac{4}{a^2 (1-|y|^2)^2} \cdot (P^{\sharp}-c) \bigg\}& = \Psi , \\
\dv v^{\sharp} & = 0,
\end{split}
\end{equation*}
where 
\[
\Psi = \frac{4}{a^2 (1-|y|^2)^2}  F^{\sharp} + \frac{16}{a^2}  \frac{1}{(1-|y|^2)^3}  (P^{\sharp} -c ) y - \frac{8}{(1-|y|^2)^2}  v^{\sharp} - \frac{4}{(1-|y|^2)}  (\partial_2v_1^{\sharp} - \partial_1 v_2^{\sharp})  y^{\perp}.
\]

By applying Lemma \ref{LinfityrescaledNEW} directly to $v^{\sharp}$, it follows that $v^{\sharp}$ satisfies the following estimate
\begin{equation}\label{straightforwardNEW}
\begin{split}\displaystyle
\big \| v^{\sharp}  \big \|_{L^{\infty}(D_0(\frac{1}{2}  \tanh (\frac{a}{2})))} \leq &
C_0 \bigg \{  \Big( \tanh \Big(\frac{a}{2}\Big) \Big)^{\frac{1}{2}}  \big \| \Psi \big \|_{L^{\frac{4}{3}}(D_0(\tanh (\frac{a}{2})))}\\
&\quad+ \displaystyle\coth\Big(\frac{a}{2}\Big) \big \| v^{\sharp} \big \|_{L^2(D_0(\tanh (\frac{a}{2})))} 
    + \big \| \nabla^{\mathbb{R}^2} v^{\sharp} \big \|_{L^2(D_0(\tanh (\frac{a}{2}) ))} \bigg \}.
\end{split}
\end{equation}
Next we use \eqref{Equation3.1NEW} in \eqref{straightforwardNEW} to get
\begin{equation}\label{supnormvsharpNEW}
\begin{split}
&\big \| v^{\sharp}  \big \|_{L^{\infty}(D_0(\frac{1}{2}  \tanh (\frac{a}{2})))} \leq 
C_0 \bigg \{  \Big ( \tanh  \Big(\frac{a}{2} \Big) \ \Big)^{\frac{1}{2}}  \big \| \Psi \big \|_{L^{\frac{4}{3}}(D_0(\tanh (\frac{a}{2})))} \\
& \qquad+ \Big( \coth \Big(\frac{a}{2}\Big) + \sinh a \Big) \big \|v \big \|_{L^2(B_O(1))} + \frac{1}{a} \cosh^2  \Big(\frac{a}{2} \Big) 
\big \| \nabla v \big \|_{L^2(B_O(1))} \bigg \}.
\end{split}
\end{equation}
Now, observe that the following estimate holds for any $y \in D_0(\tanh (\frac{a}{2}))$.
\begin{equation*}
\begin{split}
|\Psi (y)| & \leq \frac{4}{a^2}\cosh^4\Big(\frac{a}{2}\Big) \big | F^{\sharp} \big | + \frac{16}{a^2}\big | P^{\sharp}-c \big |\tanh \Big(\frac{a}{2}\Big) \cosh^6 \Big(\frac{a}{2}\Big)+ 8\cosh^4 \Big(\frac{a}{2}\Big) \big | v^{\sharp}\big |\\
&\qquad + 8 \tanh\Big (\frac{a}{2}\Big)\cosh^2\Big(\frac{a}{2}\Big) \big | \nabla^{\mathbb{R}^2}v^{\sharp}\big | \\
& = \frac{4}{a^2}\cosh^4\Big(\frac{a}{2}\Big) \big | F^{\sharp} \big | + \frac{8}{a^2}\sinh a  \cosh^4 \Big(\frac{a}{2}\Big) \big | P^{\sharp}-c \big |
+ 8 \cosh^4 \Big(\frac{a}{2}\Big) \big | v^{\sharp} \big | + 4 \sinh a \big | \nabla^{\mathbb{R}^2} v^{\sharp} \big |.
\end{split}
\end{equation*}
Then applying \eqref{FL-12NEW} and the Holder's inequality $\|f\|_{L^{\frac{4}{3}}(D_0(r))} \leq |D_0(r)|^{\frac{1}{4}}\|f\|_{L^2(D_0(r))}$, it follows directly from the above estimate that
\begin{equation*}\label{psicomplicatedone}
\begin{split}
& \big \| \Psi \big \|_{L^{\frac{4}{3}}(D_0(\tanh (\frac{a}{2})))} \\
&\ \leq 8\cosh^4\Big(\frac{a}{2}\Big) \bigg \{ \frac{1}{a^2} \big \|F^{\sharp} \big \|_{L^{\frac{4}{3}}(D_0(\tanh (\frac{a}{2})))}
+ \frac{\sinh a}{a^2} \big \| P^{\sharp} - c \big \|_{L^{\frac{4}{3}}(D_0(\tanh (\frac{a}{2})))} \\
&\quad + \big \|v^{\sharp} \big \|_{L^{\frac{4}{3}}(D_0(\tanh (\frac{a}{2})))}   \bigg \} + 4 \sinh a   \big \| \nabla^{\mathbb{R}^2} v^{\sharp} \big \|_{L^{\frac{4}{3}} (D_0(\tanh (\frac{a}{2})))}  \\
&\ \leq \frac{C_0}{a^{\frac{3}{2}}} \cosh^4 \Big(\frac{a}{2}\Big) \big \| F \big \|_{L^{\frac{4}{3}}(B_O(1))}
+ C_0 \cosh^4 \Big(\frac{a}{2}\Big) \Big ( \tanh \Big(\frac{a}{2}\Big) \Big )^{\frac{1}{2}}  
\bigg \{  \frac{\sinh a }{a^2} \big \|P^{\sharp} - c \big \|_{L^2(D_0 (\tanh (\frac{a}{2})) )} \\
& \quad+ \big \|v^{\sharp} \big \|_{L^2(D_0 (\tanh (\frac{a}{2})))}    \bigg \} + C_0 \sinh a   \Big ( \tanh \Big(\frac{a}{2}\Big) \Big )^{\frac{1}{2}}  
\big \| \nabla^{\mathbb{R}^2} v^{\sharp}  \big \|_{L^2(D_0(\tanh (\frac{a}{2})))}
\end{split}
\end{equation*}
We now use \eqref{EstimatePminusLNEW} in the above estimate to obtain

\begin{equation}\label{DelicatepsiNEW}
\begin{split}
& \big \| \Psi \big \|_{L^{\frac{4}{3}}(D_0(\tanh (\frac{a}{2})))} \\
&\quad\leq \frac{C_0}{a^{\frac{3}{2}}} \cosh^4 \Big(\frac{a}{2}\Big) \big \| F \big \|_{L^{\frac{4}{3}}(B_O(1))} \\
& \qquad\quad+ C_0 \cosh^4\Big(\frac{a}{2}\Big) \Big ( \tanh \Big(\frac{a}{2}\Big) \Big )^{\frac{1}{2}}  \sinh a   \bigg \{ \Big(1+ \tanh^2\Big(\frac{a}{2}\Big)\Big) \sinh a + \tanh \Big(\frac{a}{2}\Big) \bigg \}
\big \|v \big \|_{L^2(B_O(1))} \\
&  \qquad\quad+  C_0 \cosh^4 \Big(\frac{a}{2}\Big) \Big( \tanh \Big(\frac{a}{2}\Big) \Big )^{\frac{1}{2}}  \frac{\sinh a}{a} \cosh a  \big \| \nabla v \big \|_{L^2(B_O(1))} \\
&  \qquad\quad+ C_0 \cosh^4 \Big(\frac{a}{2}\Big) \tanh \Big(\frac{a}{2}\Big)  \frac{\sinh a}{a^{\frac{3}{2}}}  \big \| F \big \|_{L^{\frac{4}{3}}(B_O(1))} \\
& \qquad \quad+ C_0 \cosh^4 \Big(\frac{a}{2}\Big) \Big( \tanh\Big(\frac{a}{2}\Big)  \Big )^{\frac{1}{2}} \big \| v \big \|_{L^2(B_O(1))} \\
&  \qquad\quad+ C_0 \sinh a  \Big( \tanh \Big(\frac{a}{2}\Big) \Big )^{\frac{1}{2}} \big \| \nabla^{\mathbb{R}^2} v^{\sharp} \big \|_{L^2(D_0(\tanh (\frac{a}{2})))} \\
&\quad\leq \frac{C_0}{a^{\frac{3}{2}}} \cosh^4\Big (\frac{a}{2}\Big) \Big \{ 1+ \tanh \Big(\frac{a}{2}\Big) \sinh a  \Big \} \| F \big \|_{L^{\frac{4}{3}}(B_O(1))} \\
&  \qquad\quad+ C_0  \cosh^4 \Big(\frac{a}{2}\Big)\Big ( \tanh \Big(\frac{a}{2}\Big) \Big )^{\frac{1}{2}} \bigg \{ \Big( 1 + \tanh^2\Big(\frac{a}{2}\Big) \Big ) \sinh^2 a + \tanh \Big(\frac{a}{2}\Big) \sinh a + 1   \bigg \} \big \| v \big \|_{L^2(B_O(1))}  \\
&  \qquad\quad+ C_0 \cosh^4 \Big(\frac{a}{2}\Big)\Big( \tanh \Big(\frac{a}{2}\Big) \Big )^{\frac{1}{2}}  \frac{\sinh (2a)}{a}   \big \| \nabla v \big \|_{L^2(B_O(1))} \\
&  \qquad\quad+ C_0 \sinh a   \Big( \tanh \Big(\frac{a}{2}\Big) \Big )^{\frac{1}{2}} \big \| \nabla^{\mathbb{R}^2} v^{\sharp} \big \|_{L^2(D_0(\tanh (\frac{a}{2})))} .
\end{split}
\end{equation}
To simplify the above computations, we observe that the following relations hold.
\begin{equation}\label{hyptrgidentity}
\begin{split}
1+ \tanh \Big(\frac{a}{2}\Big) \sinh a & = \cosh a ,\\
\Big( 1 + \tanh^2\Big(\frac{a}{2}\Big) \Big) \sinh^2 a + \tanh \Big(\frac{a}{2}\Big) \sinh a + 1 & = \cosh a  \Big( 1 + 4 \sinh^2\Big(\frac{a}{2}\Big)  \Big) \\
& \leq 2 \cosh a  \Big( 1 + 2 \sinh^2 \Big(\frac{a}{2}\Big)\Big ) \\
& = 2 \cosh^2 a .
\end{split}
\end{equation}
So, by using the relations in \eqref{hyptrgidentity}, we can now greatly simplify estimate \eqref{DelicatepsiNEW} as follows.

\begin{equation}\label{psisimpleONENEW}
\begin{split}
\big \| \Psi \big \|_{L^{\frac{4}{3}}(D_0(\tanh (\frac{a}{2})))} \leq & \frac{C_0}{a^{\frac{3}{2}}} \cosh^4\Big(\frac{a}{2}\Big) \cosh a  \| F \big \|_{L^{\frac{4}{3}}(B_O(1))} \\
&\quad+ C_0 \cosh^4 \Big(\frac{a}{2}\Big) \Big ( \tanh \Big(\frac{a}{2}\Big) \Big)^{\frac{1}{2}}  \cosh^2 a  \big \|  v \big \|_{L^2(B_O(1))} \\
&\quad + C_0 \cosh^4 \Big(\frac{a}{2}\Big)\Big ( \tanh \Big(\frac{a}{2}\Big)\Big )^{\frac{1}{2}}  \frac{\sinh (2a)}{a}  \big \| \nabla  v \big \|_{L^2(B_O(1))} \\
&\quad + C_0 \sinh a   \Big ( \tanh \Big(\frac{a}{2}\Big) \Big)^{\frac{1}{2}}  \big \| \nabla^{\mathbb{R}^2} v^{\sharp} \big \|_{L^2(D_0(\tanh (\frac{a}{2})))} .
\end{split}
\end{equation}

Now, by applying estimate \eqref{Equation3.1NEW} in \eqref{psisimpleONENEW} we get
\begin{equation*}
\begin{split}
\big \| \Psi \big \|_{L^{\frac{4}{3}}(D_0(\tanh (\frac{a}{2})))}
\leq & \frac{C_0}{a^{\frac{3}{2}}} \cosh^4\Big(\frac{a}{2}\Big) \cosh a   \| F \big \|_{L^{\frac{4}{3}}(B_O(1))} \\
& \quad+ C_0 \cosh^4 \Big(\frac{a}{2}\Big) \Big( \tanh \Big(\frac{a}{2}\Big)\Big )^{\frac{1}{2}}   \cosh^2 a   \big \|  v \big \|_{L^2(B_O(1))} \\
&\quad + C_0 \cosh^4 \Big(\frac{a}{2}\Big) \Big( \tanh\Big (\frac{a}{2}\Big) \Big )^{\frac{1}{2}}   \frac{\sinh (2a)}{a}   \big \| \nabla  v \big \|_{L^2(B_O(1))} \\
& \quad+ C_0 \sinh a   \Big ( \tanh\Big (\frac{a}{2}\Big)\Big)^{\frac{1}{2}}   \frac{1}{a}   \cosh^2 \Big(\frac{a}{2}\Big) \big \| \nabla  v \big \|_{L^2(B_O(1))} \\
& \quad+ C_0 \sinh a   \Big( \tanh \Big(\frac{a}{2}\Big) \Big )^{\frac{1}{2}}   \sinh a   \big \|  v \big \|_{L^2(B_O(1))}  ,
\end{split}
\end{equation*}
which is equivalent to the following estimate

\begin{equation}\label{PsiFinalEstimate}
\begin{split}
 &\big \| \Psi \big \|_{L^{\frac{4}{3}}(D_0(\tanh (\frac{a}{2})))}
\leq  \frac{C_0}{a^{\frac{3}{2}}} \cosh^4\Big(\frac{a}{2}\Big) \cosh a   \| F \big \|_{L^{\frac{4}{3}}(B_O(1))} \\
&\qquad + C_0 \Big( \tanh \Big(\frac{a}{2}\Big) \Big )^{\frac{1}{2}} \Big \{ \cosh^4 \Big(\frac{a}{2}\Big) \cosh^2 a + \sinh^2 a \Big \} \big \|  v \big \|_{L^2(B_O(1))} \\
&\qquad + \frac{C_0}{a} \Big( \tanh \Big(\frac{a}{2}\Big)\Big )^{\frac{1}{2}} \sinh a   \cosh^2\Big (\frac{a}{2}\Big) \Big \{ \cosh^2 \Big(\frac{a}{2}\Big) \cosh a + 1 \Big \} \big \| \nabla  v \big \|_{L^2(B_O(1))} .
\end{split}
\end{equation}

Next, by combining \eqref{supnormvsharpNEW} with \eqref{PsiFinalEstimate}, we deduce 

\begin{equation}\label{QuiteTedious}
\begin{split}
& \big \| v^{\sharp} \big \|_{L^{\infty} (D_0 (\frac{1}{2} \tanh (\frac{a}{2})))} \\
&\quad\leq \frac{C_0}{a^{\frac{3}{2}}} \Big( \tanh \Big(\frac{a}{2}\Big) \Big)^{\frac{1}{2}} \cosh^4\Big(\frac{a}{2}\Big) \cosh a  \| F \big \|_{L^{\frac{4}{3}}(B_O(1))} \\
& \quad\qquad+ C_0  \tanh \Big(\frac{a}{2}\Big)  \Big \{ \cosh^4 \Big(\frac{a}{2}\Big) \cosh^2 a + \sinh^2 a \Big \} \big \|  v \big \|_{L^2(B_O(1))} \\
& \quad\qquad+ \frac{C_0}{a}  \tanh \Big(\frac{a}{2}\Big)  \sinh a  \cosh^2 \Big(\frac{a}{2}\Big) \bigg \{ \cosh^2\Big (\frac{a}{2}\Big) \cosh a + 1 \bigg \} \big \| \nabla  v \big \|_{L^2(B_O(1))} \\
&\quad\qquad + \Big (\coth\Big(\frac{a}{2}\Big) + \sinh a \Big ) \big \| v \big \|_{L^2 (B_O(1))} \\
& \quad\qquad+ \frac{1}{a} \cosh^2\Big(\frac{a}{2}\Big) \big \| \nabla  v \big \|_{L^2(B_O(1))} \\
&\quad= \frac{C_0}{a^{\frac{3}{2}}}\Big( \tanh \Big(\frac{a}{2}\Big)\Big)^{\frac{1}{2}} \cosh^4\Big(\frac{a}{2}\Big) \cosh a  \| F \big \|_{L^{\frac{4}{3}}(B_O(1))} \\
& \qquad+ C_0 \bigg \{ \tanh\Big (\frac{a}{2}\Big)  \Big( \cosh^4 \Big(\frac{a}{2}\Big) \cosh^2 a + \sinh^2 a \Big)  +  \coth\Big(\frac{a}{2}\Big) + \sinh a \bigg \} \big \|  v \big \|_{L^2(B_O(1))} \\
& \qquad+ \frac{C_0}{a} \cosh^2\Big(\frac{a}{2}\Big)  \bigg \{ \tanh \Big(\frac{a}{2}\Big) \sinh a \Big ( \cosh^2\Big(\frac{a}{2}\Big) \cosh a + 1  \Big ) + 1 \bigg \} \big \| \nabla  v \big \|_{L^2(B_O(1))}.
\end{split}
\end{equation}

Now, we recall the definition for the positive number $r(a) > 0$ given in \eqref{quitesimpler}
\[
r(a) = \frac{1}{a} \log \Big ( \frac{1 + 3 e^a }{3+e^a  } \Big ) .
\]
 
Note that we have the following relation, which holds on $D_0(1)$.
\begin{equation}\label{Quitestraightforward}
 | v  |_a \circ Y^{-1} =  \frac{a (1-|y|^2)}{2}| v^{\sharp}  | \leq \frac{a}{2} | v^{\sharp}  | .
\end{equation}
So, \eqref{QuiteTedious} and \eqref{Quitestraightforward} together give the following estimate:

\begin{equation*}
\begin{split}
&\big \| v \big \|_{L^{\infty} (B_O(r(a)))} \\
&\quad \leq\frac{a}{2} \big \| v^{\sharp} \big \|_{L^{\infty} (D_0 (\frac{1}{2} \tanh (\frac{a}{2})))} \\
&\quad\leq \frac{C_0}{a^{\frac{1}{2}}} \Big ( \tanh\Big (\frac{a}{2}\Big) \Big)^{\frac{1}{2}} \cosh^4\Big(\frac{a}{2}\Big) \cosh a \| F \big \|_{L^{\frac{4}{3}}(B_O(1))} \\
&\qquad + C_0  a \bigg \{ \tanh \Big(\frac{a}{2}\Big)  \Big ( \cosh^4 \Big(\frac{a}{2}\Big) \cosh^2 a + \sinh^2 a \Big )  + \coth\Big(\frac{a}{2}\Big) + \sinh a \bigg \} \big \|  v \big \|_{L^2(B_O(1))} \\
&\qquad + C_0 \cosh^2\Big(\frac{a}{2}\Big)  \bigg \{ \tanh \Big(\frac{a}{2}\Big) \sinh a  \bigg ( \cosh^2\Big(\frac{a}{2}\Big) \cosh a + 1  \bigg ) + 1 \bigg \} \big \| \nabla  v \big \|_{L^2(B_O(1))} ,
\end{split}
\end{equation*}
which is exactly estimate \eqref{Supnormestimate} as required in the conclusion of Lemma \ref{Supnormelliptic}.
\end{proof}

\section{Pointwise decay of the velocity profile}\label{section_v_decay}

Starting here, we will consider, for each $R > 0$, the exterior domain $$\Omega (R) = \mathbb{H}^2(-a^2) - \overline{B_O(R)}.$$ We first establish the following lemma.  The proof is similar to the proof of an estimate on the full hyperbolic space as it was established in \cite{CC15}, but to show it on the exterior domain, we need to use the cut-off functions from Gilbarg and Weinberger \cite{GilbargWeinberger1978}. 

\begin{lemma}\label{GlobalH1giveL2}
Let $R_0 > 0$ to be given. Consider now a divergence-free $1$-form $v \in \Lambda^1_{\sigma} (\Omega (R_0) ),$ which satisfies the following finite Dirichlet integral property.
\begin{equation}\label{FDIproperty}
\int_{\Omega (R_0)} \big | \nabla v \big |_a^2 \Vol_{\mathbb{H}^2(-a^2)} < \infty .
\end{equation}
Then, it follows that $v$ satisfies the following a priori estimate for each $R_1 > R_0$.
\begin{equation}\label{GlobalH1inducesL2}
\int_{\Omega (R_1)} \big | v \big |_a^2 \Vol_{\mathbb{H}^2(-a^2)} \leq \frac{2}{a^2}  \bigg \{  2 + \frac{18}{a^2} \bigg ( \frac{4}{(R_1-R_0)}   \bigg )^2        \bigg \}\int_{\Omega (R_0)} \big | \nabla v \big |_a^2 \Vol_{\mathbb{H}^2(-a^2)}.
\end{equation}
\end{lemma}

\begin{proof}
Let $v \in \Lambda^1 (\Omega (R_0) ) \cap C^0 (\overline{\Omega (R_0) })$ satisfy \eqref{FDIproperty}. We now consider a cut off function $\varphi_1 \in C^{\infty} ([0,\infty ))$ which satisfies the following.
\begin{equation}\label{CutoffpropertyONE}
\begin{split}\displaystyle
& \chi_{_{[R_1 , \infty )}} \leq \varphi_1 \leq  \chi_{_{[\frac{R_0 + R_1}{2} , \infty )}}, \\
& \big |\varphi_1' \big | \leq \frac{4}{(R_1-R_0)}  \chi_{_{[\frac{R_0 + R_1}{2} , R_1 ]}} .
\end{split}
\end{equation}

Next, we also need another cut off function $\varphi_2 \in C_c^{\infty} ([0,2 ))$ such that
\begin{equation}\label{CutoffpropertyTWO}
\begin{split}
& \chi_{_{[0,1]}} \leq \varphi_2 \leq \chi_{_{[0,2]}} ,\\
& \big | \varphi_2' \big | \leq 2 \cdot \chi_{_{[1,2]}} .
\end{split}
\end{equation}
Now, let us select a fixed $R_1 > R_0$, and let $R \geq \max \{ R_1 , 1 \}$, with respect to which we consider the cut-off function
$\eta_{_{R}} \in C_c^{\infty} (\Omega (R_0))$ defined as follows.

\begin{equation*}
\eta_{_{R}} (x) = \varphi_1 (\rho (x)) \cdot \varphi_2 \Big(\frac{\rho (x)}{R}\Big) ,
\end{equation*}
where $\rho (x)$ is the geodesic distance in $\mathbb{H}^2(-a^2)$  from $O$ to $x$. Then by the definition of $\eta_{_{R}}$  
\begin{equation}\label{GradEtaR}
\nabla \eta_{_{R}} (x) = \varphi_1'(\rho (x))  \nabla \rho (x) \varphi_2\Big (\frac{\rho (x)}{R}\Big) + \varphi_1(\rho (x))\varphi_2' \Big(\frac{\rho (x)}{R}\Big) \frac{\nabla \rho (x)}{R} .
\end{equation}
Since $\abs{\nabla \rho(x)}_a=1$, \eqref{CutoffpropertyONE}, \eqref{CutoffpropertyTWO} and \eqref{GradEtaR} give
\begin{equation}\label{estimateGradEtaR}
\big | \nabla \eta_{_{R}} (x) \big |_a \leq \frac{4}{(R_1 -R_0)} + \frac{2}{R},
\end{equation}
for all $x \in \Omega (R_0)$.
We now recall the Bochner-Weitzenb\"ock formula
\[
\nabla^\ast\nabla=\dd^\ast \dd+\dd\dd^\ast-\Ric,
\]
so for the divergence-free $1$-form $v$ on $\Omega(R_0)$ we get
\begin{equation*}
\nabla^* \nabla v = \dd^*\dd v  + a^2 v ,
\end{equation*}
from which we yield  
\begin{equation}\label{VERYTRIVIAL}
\int_{\Omega (R_0)} g ( \nabla^* \nabla v , \eta_{_{R}}^2 v   ) \Vol_{\mathbb{H}^2(-a^2)} = \int_{\Omega (R_0)} g (\dd^* \dd v , \eta_{_{R}}^2  v ) \Vol_{\mathbb{H}^2(-a^2)}
+ a^2 \int_{\Omega (R_0)} \eta_{_{R}}^2 \big | v \big |_a^2  \Vol_{\mathbb{H}^2(-a^2)}.
\end{equation}
Since $\eta_{_{R}}^2 v \in \Lambda_c^1 (\Omega (R_0))$, we can do integration by parts as follows.
\begin{equation}\label{Integrationbypart1}
\begin{split}
& \int_{\Omega (R_0)} g ( \nabla^* \nabla v , \eta_{_{R}}^2  v   ) \Vol_{\mathbb{H}^2(-a^2)} \\
= & \int_{\Omega (R_0)} g (\nabla v , \nabla (\eta_{_{R}}^2  v  ) )   \Vol_{\mathbb{H}^2(-a^2)} \\
= & \int_{\Omega (R_0)} g ( \nabla v , 2 \eta_{_{R}} \dd \eta_{_{R}} \otimes v  ) \Vol_{\mathbb{H}^2(-a^2)} + \int_{\Omega (R_0)} \eta_{_{R}}^2 \big | \nabla v \big |_a^2 \Vol_{\mathbb{H}^2(-a^2)} .
\end{split}
\end{equation}
In the same way, we have
\begin{equation}\label{Integrationbypart2}
\begin{split}
\int_{\Omega (R_0)} g ( \dd^* \dd v , \eta_{_{R}}^2 v ) \Vol_{\mathbb{H}^2(-a^2)}
& = \int_{\Omega (R_0)} g ( \dd v , 2 \eta_{_{R}} \dd \eta_{_{R}} \wedge v   ) \Vol_{\mathbb{H}^2(-a^2)} + \int_{\Omega (R_0)} \eta_{_{R}}^2 \big | \dd v \big |_a^2 \Vol_{\mathbb{H}^2(-a^2)}.
\end{split}
\end{equation}
Recall that we have the following standard estimates
\begin{equation*}
\begin{split}
\big | \dd v \big |_a & \leq \big | \nabla v  \big |_a , \\
\big | \dd \eta_{_{R}} \wedge v  \big |_a & \leq 2 \big | \nabla \eta_{_{R}} \big |_a \cdot \big | v \big |_a .
\end{split}
\end{equation*}
So, \eqref{VERYTRIVIAL}, \eqref{Integrationbypart1}, and \eqref{Integrationbypart2} together give the following estimate.
\begin{equation*}
\begin{split}
& a^2 \int_{\Omega (R_0)} \eta_{_{R}}^2 \big | v \big |_a^2 \Vol_{\mathbb{H}^2(-a^2)} \\
&\quad\leq 2 \int_{\Omega (R_0)} \eta_{_{R}}^2 \big | \nabla v \big |_a^2 \Vol_{\mathbb{H}^2(-a^2)} + 6 \int_{\Omega (R_0)} \eta_{_{R}} \big |\nabla \eta_{_{R}} \big |_a  
\big | v \big |_a   \big |\nabla v \big |_a  \Vol_{\mathbb{H}^2(-a^2)} \\
&\quad\leq 2 \int_{\Omega (R_0)}  \big | \nabla v \big |_a^2 \Vol_{\mathbb{H}^2(-a^2)} + 3 \varepsilon \int_{\Omega (R_0)} \eta_{_{R}}^2 \big | v \big |_a^2 \Vol_{\mathbb{H}^2(-a^2)} + \frac{3}{\varepsilon} \int_{\Omega (R_0)} \big |\nabla \eta_{_{R}} \big |_a^2   \big | \nabla v \big |_a^2  \Vol_{\mathbb{H}^2(-a^2)} .
\end{split}
\end{equation*}
By taking $\varepsilon = \frac{a^2}{6}$ in the above estimate, it follows through applying \eqref{estimateGradEtaR} that the following estimate holds
\begin{equation*}
\begin{split}
& \frac{a^2}{2} \int_{\Omega (R_0)} \eta_{_{R}}^2 \big | v \big |_a^2 \Vol_{\mathbb{H}^2(-a^2)} \\
&\quad\leq 2 \int_{\Omega (R_0)}  \big | \nabla v \big |_a^2 \Vol_{\mathbb{H}^2(-a^2)} + \frac{18}{a^2} \bigg (  \frac{4}{(R_1-R_0)} + \frac{2}{R} \bigg )^2  \int_{\Omega (R_0)}  \big | \nabla v \big |_a^2 \Vol_{\mathbb{H}^2(-a^2)} \\
&\quad=  \bigg \{ 2 +  \frac{18}{a^2} \bigg (  \frac{4}{(R_1-R_0)} + \frac{2}{R} \bigg )^2   \bigg \} \int_{\Omega (R_0)}  \big | \nabla v \big |_a^2 \Vol_{\mathbb{H}^2(-a^2)} .
\end{split}
\end{equation*}
By taking $R$ to $\infty$ we get the estimate \eqref{GlobalH1inducesL2} as needed.

\end{proof}

We are now ready to establish Theorem \ref{VelocityDecayThm}.
\subsection{Proof of Theorem \ref{VelocityDecayThm}}

\begin{proof}
Let $v$ be as stated in the hypotheses. Since $v$ satisfies  \eqref{FDIforVelDecay}, we can apply Lemma \ref{GlobalH1giveL2} to get
\begin{equation*}
\int_{\Omega (R_0 + 1)} \big | v \big |_a^2 \Vol_{\mathbb{H}^2(-a^2)} \leq \frac{2}{a^2} \Big ( 2 + \frac{18\cdot 16}{a^2}   \Big ) \int_{\Omega (R_0)}
\big | \nabla v \big |_a^2 \Vol_{\mathbb{H}^2(-a^2)}.
\end{equation*}
 
Now, we take a smooth function $\varphi \in C^{\infty} ([0,\infty ))$ which satisfies the following properties.
\begin{equation*}
\begin{split}
& \chi_{_{[R_0 + 2 ,\infty )}} \leq \varphi \leq \chi_{_{[R_0 +1 , \infty )}} ,\\
& \big | \varphi ' \big | \leq 2    \chi_{_{[R_0 + 1 ,R_0 + 2 ]}} .
\end{split}
\end{equation*}
Next, we consider the radially symmetric cut-off function $\eta \in C^{\infty} (\mathbb{H}^2(-a^2))$  
\begin{equation*}
\eta (x) = \varphi (\rho (x)).
\end{equation*}
Let $w = \eta   v$. Notice that the support of $\eta$ lies in $\Omega (R_0 + 1)$. This tells us that $w$ can be regarded as a globally defined smooth $1$-form on the whole space-form $\mathbb{H}^2(-a^2)$ and that the support of $w$ also lies in $\Omega (R_0 + 1)$. That is, we have $w \in \Lambda^1 (\mathbb{H}^2(-a^2))$. Then, $w$ clearly satisfies the following properties.
\begin{equation*}
\begin{split}
\int_{\mathbb{H}^2(-a^2)} \big | w \big |_a^2 \Vol_{\mathbb{H}^2(-a^2)} & \leq \int_{\Omega (R_0 + 1)} \big | v \big |_a^2 \Vol_{\mathbb{H}^2(-a^2)} , \\
\int_{\mathbb{H}^2(-a^2)} \big | \nabla w \big |_a^2 \Vol_{\mathbb{H}^2(-a^2)} & \leq 2 \int_{\Omega (R_0 + 1)} \big | \dd \eta \big |_a^2   \big | v \big |_a^2 \Vol_{\mathbb{H}^2(-a^2)} + 2 \int_{\Omega (R_0 + 1)} \big | \nabla v \big |_a^2 \Vol_{\mathbb{H}^2(-a^2)} \\
& \leq 8\int_{\Omega (R_0 + 1)} \big | v \big |_a^2 \Vol_{\mathbb{H}^2(-a^2)} + 2 \int_{\Omega (R_0)} \big | \nabla v \big |_a^2 \Vol_{\mathbb{H}^2(-a^2)} .
\end{split}
\end{equation*}
So $w \in H^1(\mathbb{H}^2(-a^2))$. The hyperbolic Ladyzhenskaya inequality gives (see for example \cite{CC13})
\begin{equation*}
\big \| w \big \|_{L^4(\mathbb{H}^2(-a^2))} \leq C_a \Big \{ \big \| w \big \|_{L^2 (\mathbb{H}^2(-a^2))} + \big \| \nabla w \big \|_{L^2 (\mathbb{H}^2(-a^2))} \Big \} .
\end{equation*}
Now, notice that $w(x) = v (x)$ holds for all $x \in \Omega (R_0 + 2)$. So, we have the following straightforward estimate.
\begin{equation*}
\begin{split}
\int_{\Omega (R_0 +2)} \big | \nabla_v v \big |_a^{\frac{4}{3}} \Vol_{\mathbb{H}^2(-a^2)} & \leq \int_{\mathbb{H}^2(-a^2)} \big | \nabla_w w  \big |^{\frac{4}{3}} \Vol_{\mathbb{H}^2(-a^2)} \\
& \leq \int_{\mathbb{H}^2(-a^2)} \big | w \big |_a^{\frac{4}{3}}   \big | \nabla w  \big |_a^{\frac{4}{3}}  \Vol_{\mathbb{H}^2(-a^2)} \\
& \leq \Big ( \int_{\mathbb{H}^2(-a^2)} \big | w \big |_a^4 \Vol_{\mathbb{H}^2(-a^2)} \Big )^{\frac{1}{3}}  \Big (  \int_{\mathbb{H}^2(-a^2)} \big |\nabla w \big |_a^2  \Vol_{\mathbb{H}^2(-a^2)}  \Big )^{\frac{2}{3}} .
\end{split}
\end{equation*}
 Now, let us take an arbitrary point $x_0 \in \Omega (R_0 + 8)$. Then, applying Lemma \ref{Supnormelliptic}  over the geodesic ball $B_{x_0} (1)$, we immediately obtain  \begin{equation}\label{SecondbeforeLast}
\big \|v\big \|_{L^{\infty} (B_{x_0} (r(a)) )} \leq C_0 \Big \{ A_1(a) \big \| \nabla_v v \big \|_{L^{\frac{4}{3}} (B_{x_0}(1))}
+ A_2(a) \|v \|_{L^2(B_{x_0}(1))} + A_3 (a) \big \| \nabla v \big \|_{L^2(B_{x_0}(1))} \Big \}.
\end{equation}
Now, since we have $\nabla_v v \in L^{\frac{4}{3}}(\Omega (R_0 +2))$, $v \in L^2(\Omega (R_0 + 1 )) $ and $\nabla v \in L^2(\Omega (R_0))$, it follows that we have the following decay properties.
\begin{equation}\label{EsayDECAY}
\begin{split}
\lim_{\rho(x_0) \rightarrow \infty } \big \| \nabla_v v \big \|_{L^{\frac{4}{3}}(B_{x_0}(1))} & = 0 , \\
\lim_{\rho(x_0) \rightarrow \infty } \big \| v \big \|_{L^2(B_{x_0}(1))} & = 0 ,\\
\lim_{\rho(x_0) \rightarrow \infty } \big \| \nabla v \big \|_{L^2(B_{x_0}(1))} & = 0.
\end{split}
\end{equation}
So, by combining \eqref{SecondbeforeLast} with \eqref{EsayDECAY}, we have
\begin{equation}
\lim_{\rho (x_0) \rightarrow \infty } \big \|v\big \|_{L^{\infty} (B_{x_0} (r(a)) )} = 0.
\end{equation}
This completes the proof of Theorem \ref{VelocityDecayThm}.
\end{proof}

\section{About the vorticity.}\label{section_vorticity}

In this section, we show vorticity is in $H^1$, which is used to establish the rate of decay in the far range, Theorem \ref{FinalThmforExpdecay}.  As in the last section, we will use the notation $\Omega (R) = \mathbb{H}^2(-a^2) - \overline{B_O(R)}$, for any $R > 0$.   We start with the $H^1$ property.

\subsection{ $H^1$-property of the vorticity.} The statement of the following theorem and the proof is based on Gilbarg and Weinberger \cite[Lemma 2.3]{GilbargWeinberger1978}. 

\begin{thm}\label{L2propertyvorticity}
Let $R_0 > 0$. Consider a smooth $1$-form $v \in \Lambda^1_{\sigma} (\Omega (R_0)),$ which satisfies the following stationary Navier-Stokes equation on $\Omega (R_0)$, with $P$ to be some smooth function on $\Omega (R_0)$.
\begin{equation}\label{NSforvorticity}
\begin{split}
2 \Def^* \Def v + \nabla_v v + \dd P & = 0 ,\\
\dd^* v & = 0.
\end{split}
\end{equation}
Suppose that $v$ also satisfies  
\begin{equation}\label{FDNpropertyVort}
\int_{\Omega (R_0)} \big | \nabla v \big |_a^2 \Vol_{\mathbb{H}^2(-a^2)} < \infty.
\end{equation}
Consider $\omega \in C^{\infty} (\Omega (R_0))$ to be the function defined as follows.
\begin{equation}\label{DefofVorticity}
\omega = * \dd v .
\end{equation}
Then, for any $R_1 > R_0$, the following a priori estimate holds.
\begin{equation}\label{Vorticity2.23}
\begin{split}
& \int_{\Omega (R_1)} \big | \nabla \omega \big |_a^2 \Vol_{\mathbb{H}^2(-a^2)} \\
&\quad\leq 2a^2 \int_{\Omega (R_0)} \big | \nabla v \big |_a^2 \Vol_{\mathbb{H}^2(-a^2)} + C(a,R_0 , R_1 ) \int_{A\big(\frac{R_0 + R_1}{2} ,R_1\big )} \big | \omega \big |_a^2 \big ( 1 + \big | v \big |_a \big ) \Vol_{\mathbb{H}^2(-a^2)},
\end{split}
\end{equation}
where $A\big(\frac{R_0 + R_1}{2} , R_1\big) = \big\{x \in \mathbb{H}^2(-a^2) : \frac{R_0 + R_1}{2} < \rho (x) < R_1  \big\}$ .
\end{thm}

\begin{proof}
Here, we closely follow the main ideas of a lemma by Gilbarg and Weinberger \cite[Lemma 2.3]{GilbargWeinberger1978}. So, we take an arbitrary positive number $L > 0$, which plays the role of the level of truncation. With respect to $L$, we consider the associated function $h_{_{L}} \in C^1 (\mathbb{R})$ which is defined as follows.
\begin{equation}\label{DefofhL}
h_{_{L}}(\lambda ) = \lambda^2   \chi_{_{\{|\lambda| \leq L \}} } + \big \{ 2 L ( |\lambda | -  L ) + L^2 \big \}   \chi_{_{ \{ |\lambda | > L \}}  } .
\end{equation}

Hence, it follows that
\begin{equation*}
h_{_{L}}'(\lambda ) = -2L   \chi_{_{\{ \lambda < -L \}} } + 2 \lambda   \chi_{_{ \{ |\lambda | \leq L \} }} + 2L    \chi_{_{\{ \lambda > L   \}}} ,
\end{equation*}
and that
\begin{equation*}
h_{_{L}}''(\lambda ) = 2\chi_{_{\{ |\lambda | < L \}}} .
\end{equation*}

Now, take any $R_1 > R_0$.
As in the proof of Lemma \ref{GlobalH1giveL2}, we consider the very same cut off functions $\varphi_1 \in C^{\infty} ([0,\infty )) $ , $\varphi_2 \in C^{\infty} ( [0,2 ))$ which are characterized by conditions \eqref{CutoffpropertyONE} and \eqref{CutoffpropertyTWO} respectively.  We also use the same $\eta_R$, for each $R \geq \max \{ R_1 , 1 \}$, which was given by
\begin{equation}\label{DefETA_R}
\eta_{_{R}} (x) = \varphi_1 (\rho (x) )  \varphi_2 \Big ( \frac{\rho (x)}{R} \Big ).
\end{equation}
Now, notice that we have
\begin{equation*}
\eta_{_{R}}   \chi_{_{\{ R \leq \rho \leq 2 R  \}}} = \varphi_2 \Big (\frac{\rho}{R}\Big )   \chi_{_{ \{ R \leq \rho \leq 2 R \}} }.
\end{equation*}
Hence, it follows  
\begin{equation}\label{Vorticity2.7}
\nabla \eta_{_{R}}   \chi_{_{\{ R \leq \rho \leq 2R \}}} = \varphi_2' \Big ( \frac{\rho}{R} \Big )   \frac{\nabla \rho }{R}   \chi_{_{\{ R \leq \rho \leq 2R \}}} ,
\end{equation}
which directly gives
\begin{equation}\label{Vorticity2.8}
\big |\nabla \eta_{_{R}} \big |   \chi_{_{\{ R \leq \rho \leq 2R \}}} \leq \frac{2}{R}   \chi_{_{\{ R \leq \rho \leq 2R \}}} .
\end{equation}
Identity \eqref{Vorticity2.7} leads to
\begin{equation*}
\Delta \eta_{_{R}}   \chi_{_{\{ R \leq \rho \leq 2R \}}} = \bigg ( \varphi_2'' \Big(\frac{\rho}{R} \Big)   \frac{1}{R^2} + \varphi_2'  \Big(\frac{\rho}{R} \Big)   \frac{\Delta \rho}{R} \bigg )   \chi_{_{\{ R \leq \rho \leq 2R \}}} .
\end{equation*}
Since $\Delta \rho = a \coth (a \rho )$ holds on $\mathbb{H}^2(-a^2)-{O}$, it follows from the above identity that
\begin{equation}\label{Vorticity2.9}
\begin{split}
\big | \Delta \eta_{_{R}} \big |   \chi_{_{\{ R \leq \rho \leq 2R \}}} & \leq \bigg ( \frac{1}{R^2}   \big \|\varphi_2'' \big \|_{L^{\infty}([0,2))} +
 \frac{2a\coth (a\rho )}{R} \bigg )    \chi_{_{\{ R \leq \rho \leq 2R \}}} \\
 & \leq  \bigg ( \frac{1}{R^2}   \big \|\varphi_2'' \big \|_{L^{\infty}([0,2))} +
 \frac{2a\coth (a R )}{R} \bigg )    \chi_{_{\{ R \leq \rho \leq 2R \}}} ,
\end{split}
\end{equation}
where the second inequality follows from the fact that $\coth (t)$ is monotone decreasing in $t \in (0,\infty )$.\\

Recall $\omega=\ast \dd v$, which by definition is equivalent to 
\begin{equation*}
\dd v = \omega \Vol_{\mathbb{H}^2(-a^2)}.
\end{equation*}
So using the estimate
\begin{equation*}
\big | \dd v \big |_a \leq  \big | \nabla v \big |_a, 
\end{equation*}
and the finite Dirichlet-norm property \eqref{FDNpropertyVort} we get
\begin{equation}\label{Vorticity2.10}
\begin{split}
\int_{\Omega (R_0)} \big | \omega \big |^2 \Vol_{\mathbb{H}^2(-a^2)} & = \int_{\Omega (R_0)} \big | \dd v \big |_a^2 \Vol_{\mathbb{H}^2(-a^2)}  \leq  \int_{\Omega (R_0)} \big | \nabla v \big |_a^2 \Vol_{\mathbb{H}^2(-a^2)}.
\end{split}
\end{equation}
Now, as in the paper by Gilbarg and Weinberger, we carry out the following computation
\begin{equation}\label{Vorticity2.11}
\begin{split}
 {2}\eta_{_{R}}   \chi_{_{\{ |\omega | < L \}}}   \big | \nabla \omega \big |_a^2
=&  \eta_{_{R}}   h_{_{L}}''(\omega )  g( \nabla \omega ,  \nabla \omega) \\
= & \eta_{_{R}} g( \nabla (h_{_{L}}'(\omega ))  ,\nabla \omega) \\
= & \dv \big \{ \eta_{_{R}}   h_{_{L}}'(\omega ) \nabla \omega \big \} - g(\nabla \eta_{_{R}} ,  h_{_{L}}'(\omega ) \nabla \omega) - \eta_{_{R}}   h_{_{L}}'(\omega ) \Delta \omega \\
= & \dv \big \{ \eta_{_{R}}   h_{_{L}}'(\omega ) \nabla \omega \big \} - g(\nabla \eta_{_{R}}  , \nabla \big ( h_{_{L}}(\omega ) \big )) - \eta_{_{R}}   h_{_{L}}'(\omega ) \Delta \omega .
\end{split}
\end{equation}

Next, taking $* \dd$ on both sides of the first line of   \eqref{NSforvorticity}, we obtain the following equation satisfied by the vorticity function $\omega$ on $\Omega (R_0)$.

\begin{equation}\label{Vorticity2.12}
-\Delta \omega + 2a^2 \omega + g(v , \nabla \omega) = 0 .
\end{equation}
By using \eqref{Vorticity2.12}, it follows from identity \eqref{Vorticity2.11} that
\begin{equation}\label{Vorticity2.13}
\begin{split}
&{2}\int_{\Omega (R_0)} \big | \nabla \omega \big |_a^2 \eta_{_{R}} \chi_{_{\{ |\omega | < L \}} } \Vol_{\mathbb{H}^2(-a^2)} \\
= & \int_{\Omega (R_0)} -g(\nabla \eta_{_{R}}  , \nabla \big ( h_{_{L}}(\omega )\big ) )- \eta_{_{R}}   h_{_{L}}'(\omega )   \big ( 2a^2 \omega + g(v,  \nabla \omega \big )) \Vol_{\mathbb{H}^2(-a^2)} \\
= & \int_{\Omega (R_0)}  -g(\nabla \eta_{_{R}} ,  \nabla \big ( h_{_{L}}(\omega )\big )) - \eta_{_{R}} g( v  , \nabla \big ( h_{_{L}}(\omega )\big ) )- 2a^2 h_{_{L}}'(\omega )  \eta_{_{R}}   \omega \Vol_{\mathbb{H}^2(-a^2)} \\
= & \int_{\Omega (R_0)} \big ( \Delta \eta_{_{R}} +g( v,   \nabla \eta_{_{R}} \big )   h_{_{L}}(\omega ) )\Vol_{\mathbb{H}^2(-a^2)} - 2a^2 \int_{\Omega (R_0)} h_{_{L}}'(\omega )   \eta_{_{R}}   \omega \Vol_{\mathbb{H}^2(-a^2)} ,
\end{split}
\end{equation}
where the last equality follows from integration by parts and the divergence-free property $\dd^* v = 0$ of $v$. \\

Since we know the following straightforward estimate
\begin{equation*}
\big | h_{_{L}} ' (\omega ) \big | \leq 2\min \{  |\omega | , L \} ,
\end{equation*}
it follows, by taking \eqref{Vorticity2.10} into our account, that we have
\begin{equation}\label{Vorticity2.14}
\begin{split}
\bigg | \int_{\Omega (R_0)} h_{_{L}}'(\omega )   \eta_{_{R}}  \omega \Vol_{\mathbb{H}^2(-a^2)}   \bigg | & \leq 2 \int_{\Omega (R_0)} \big | \omega \big |^2  \Vol_{\mathbb{H}^2(-a^2)} \leq 2 \int_{\Omega(R_0)} \big | \nabla v \big |_a^2 \Vol_{\mathbb{H}^2(-a^2)} .
\end{split}
\end{equation}
Recall 
\begin{equation*}
A(r_1 , r_2) = \big \{ x \in \mathbb{H}^2(-a^2) : r_1 \leq \rho(x) \leq r_2 \big \},
\end{equation*}

and observe that we have
\begin{equation}\label{Vorticity2.15}
\begin{split}
& \bigg | \int_{\Omega (R_0)} \big ( \Delta \eta_{_{R}} + g(v,   \nabla \eta_{_{R}} )\big )   h_{_{L}}(\omega ) \Vol_{\mathbb{H}^2(-a^2)} \bigg | \\
&\quad \leq \int_{A\big( \frac{R_0 + R_1}{2} , R_1   \big)} \big | h_{_{L}}(\omega ) \big |   \Big \{ \big |\Delta \big (\varphi_1(\rho)\big ) \big | + \big | v \big |_a   \big | \nabla \big ( \varphi_1 (\rho ) \big ) \big |_a    \Big \}  \Vol_{\mathbb{H}^2(-a^2)} \\
&\qquad\quad  + \int_{A(R, 2R)} \big | h_{_{L}}(\omega ) \big |  \Big \{  \big | \Delta \eta_{_{R}} \big | + \big | g(v,  \nabla \eta_{_{R}} )\big |  \Big \} \Vol_{\mathbb{H}^2(-a^2)} .
\end{split}
\end{equation}

By using the obvious relation $0 \leq h_{_{L}}(\lambda ) \leq \lambda^2$, the first integral which appears on the right-hand side of \eqref{Vorticity2.15} will be controlled as follows.
\begin{equation}\label{simplesimplesimple1000}
\begin{split}
& \int_{A\big( \frac{R_0 + R_1}{2} , R_1   \big)} \big | h_{_{L}}(\omega ) \big | \Big \{ \big |\Delta \big ( \varphi_1 (\rho ) \big ) \big | + \big | v \big |_a \big | \nabla \big ( \varphi_1(\rho ) \big ) \big |_a    \Big \}  \Vol_{\mathbb{H}^2(-a^2)} \\
&\quad\leq C(a, R_0 , R_1)  \int_{A\big( \frac{R_0 + R_1}{2} , R_1   \big)} \big | \omega \big |^2  \big ( 1 + \big | v \big |_a \big )  \Vol_{\mathbb{H}^2(-a^2)}  < \infty ,
\end{split}
\end{equation}
where the absolute constant $C(a, R_0, R_1)$ is just
\begin{equation*}
C(a, R_0 , R_1) = \big \| \Delta \big ( \varphi_1 (\rho ) \big ) \big \|_{L^{\infty} (\mathbb{H}^2(-a^2))} + \big \| \nabla \big ( \varphi_1 (\rho ) \big ) \big \|_{L^{\infty} (\mathbb{H}^2(-a^2))} .
\end{equation*}

Next, we have to prove that the second integral which appears on the right-hand side of \eqref{Vorticity2.15} tends to $0$ as $R$ goes to infinity. We now achieve this as follows. First, notice that we have the following straightforward estimate, which holds for any $\lambda \in \mathbb{R}$.

\begin{equation}\label{estimateofh_{_{L}}}
\big | h_{_{L}}( \lambda ) \big | \leq \min \big \{ \lambda ^2 , 2 L |\lambda | \big  \} .
\end{equation}

So, by combining \eqref{Vorticity2.8}, \eqref{Vorticity2.9} with \eqref{estimateofh_{_{L}}}, it follows that

\begin{equation}\label{Vorticity2.17}
\begin{split}
& \int_{A(R, 2R)} \big | h_{_{L}}(\omega ) \big | \Big \{  \big | \Delta \eta_{_{R}} \big | + \big | g(v ,\nabla \eta_{_{R}} )\big |  \Big \} \Vol_{\mathbb{H}^2(-a^2)} \\
&\quad\leq \bigg ( \frac{1}{R^2}  \big \| \varphi_2'' \big \|_{L^{\infty} ([0,2))} + \frac{2a \coth (aR)}{R}   \bigg ) \int_{A(R, 2R)} \big | \omega \big |^2
\Vol_{\mathbb{H}^2(-a^2)} \\
& \qquad\qquad\qquad\qquad+ \frac{4}{R} \int_{A(R, 2R)} L  |\omega | \big | v \big |_a \Vol_{\mathbb{H}^2(-a^2)} \\
&\quad\leq \bigg ( \frac{1}{R^2} \big \| \varphi_2'' \big \|_{L^{\infty} ([0,2))} + \frac{2a \coth (aR)}{R}   \bigg ) \int_{A(R, 2R)} \big | \omega \big |^2
\Vol_{\mathbb{H}^2(-a^2)}  \\
&\qquad\qquad\qquad\qquad + \frac{4L}{R} \big \| \omega \big \|_{L^2(A(R, 2R))} \big \| v \big \|_{L^2(A(R, 2R))} .
\end{split}
\end{equation}
Independently, we also observe that since we have $\omega \in L^2(\Omega (R_0))$ and $ v \in L^2(\Omega (R_0))$, it must hold that
\begin{equation*}
\lim_{R \rightarrow \infty } \bigg \{  \int_{A(R, 2R)} |\omega |^2 \Vol_{\mathbb{H}^2(-a^2)} + \int_{A(R, 2R)} |v|^2 \Vol_{\mathbb{H}^2(-a^2)}  \bigg \} = 0 .
\end{equation*}
So, it follows from \eqref{Vorticity2.17} that we have
\begin{equation}\label{Vorticity2.18}
\lim_{R\rightarrow \infty} R \int_{A(R, 2R)} \big | h_{_{L}}(\omega ) \big |  \Big \{  \big | \Delta \eta_{_{R}} \big | + \big | g(v ,\nabla \eta_{_{R}} )\big |  \Big \} \Vol_{\mathbb{H}^2(-a^2)} = 0.
\end{equation}
Actually, \eqref{Vorticity2.18} immediately implies the following weaker conclusion.
\begin{equation}\label{Vorticity2.19}
\lim_{R\rightarrow \infty} \int_{A(R, 2R)} \big | h_{_{L}}(\omega ) \big | \Big \{  \big | \Delta \eta_{_{R}} \big | + \big | g(v , \nabla \eta_{_{R}} )\big |  \Big \} \Vol_{\mathbb{H}^2(-a^2)} = 0.
\end{equation}
\eqref{Vorticity2.19} allows us to pass to the limit on both sides of \eqref{Vorticity2.15}, and then using \eqref{simplesimplesimple1000} we have
\begin{equation}\label{Vorticity2.20}
\begin{split}
& \limsup_{R\rightarrow \infty} \bigg |  \int_{\Omega (R_0)} \big ( \Delta \eta_{_{R}} + g(v , \nabla \eta_{_{R}} ) \big )  h_{_{L}}(\omega ) \Vol_{\mathbb{H}^2(-a^2)}   \bigg | \\
&\quad\leq C(a, R_0 , R_1) \int_{A \big(\frac{R_0 + R_1}{2} , R_1 \big )} |\omega |^2  \big ( 1 +  \big | v \big |_a \big ) \Vol_{\mathbb{H}^2(-a^2)}.
\end{split}
\end{equation}
Now, from the definition of $\varphi_1$ and $\eta_R$ we get
\begin{align*}
&2\int_{\Omega (R_1)} \big | \nabla \omega \big |_a^2 \chi_{_{\{ |\omega | < L \}}}  \Vol_{\mathbb{H}^2(-a^2)} \\
&\quad\leq  {2}\int_{\Omega (R_0)} \big | \nabla \omega \big |_a^2 \varphi_1 (\rho )    \chi_{_{\{ |\omega | < L \}}}  \Vol_{\mathbb{H}^2(-a^2)} \\
&\quad= 2 \lim_{R \rightarrow \infty } \int_{\Omega (R_0)} \big | \nabla \omega \big |_a^2 \eta_{_{R}}    \chi_{_{\{ |\omega | < L \}}} \Vol_{\mathbb{H}^2(-a^2)},
\end{align*}
so by means of \eqref{Vorticity2.14} and \eqref{Vorticity2.20}, we now pass to the limit on both sides of \eqref{Vorticity2.13} to deduce 
\begin{equation}\label{Vorticity2.21}
\begin{split}
& 2\int_{\Omega (R_1)} \big | \nabla \omega \big |_a^2 \chi_{_{\{ |\omega | < L \}}}  \Vol_{\mathbb{H}^2(-a^2)} \\
&\quad\leq2 \lim_{R \rightarrow \infty } \int_{\Omega (R_0)} \big | \nabla \omega \big |_a^2 \eta_{_{R}}    \chi_{_{\{ |\omega | < L \}}} \Vol_{\mathbb{H}^2(-a^2)} \\
&\quad\leq  4a^2 \int_{\Omega (R_0)} \big | \nabla v \big |_a^2 \Vol_{\mathbb{H}^2(-a^2)} + \limsup_{R\rightarrow \infty}  \bigg  |  \int_{\Omega (R_0)} \big ( \Delta \eta_{_{R}} + v   \nabla \eta_{_{R}}  \big )   h_{_{L}}(\omega ) \Vol_{\mathbb{H}^2(-a^2)}    \bigg  | \\
&\quad\leq  4a^2 \int_{\Omega (R_0)} \big | \nabla v \big |_a^2 \Vol_{\mathbb{H}^2(-a^2)} + C(a, R_0 , R_1) \int_{A \big(\frac{R_0 + R_1}{2} , R_1  \big)} |\omega |^2   \big ( 1 +  \big | v \big |_a \big ) \Vol_{\mathbb{H}^2(-a^2)} .
\end{split}
\end{equation}
Finally, we take $L \rightarrow \infty$ on both sides of \eqref{Vorticity2.21}  to obtain
\begin{equation*}
\begin{split}
& \int_{\Omega (R_1)} \big | \nabla \omega \big |_a^2 \Vol_{\mathbb{H}^2(-a^2)} \\
&\quad\leq 2a^2 \int_{\Omega (R_0)} \big | \nabla v \big |_a^2 \Vol_{\mathbb{H}^2(-a^2)} + C(a, R_0 , R_1) \int_{A \big(\frac{R_0 + R_1}{2} , R_1 \big )} |\omega |^2   \big ( 1 +  \big | v \big |_a \big ) \Vol_{\mathbb{H}^2(-a^2)} ,
\end{split}
\end{equation*}
which is exactly estimate \eqref{Vorticity2.23} as required in the statement of Theorem \ref{L2propertyvorticity}.

\end{proof}

\subsection{About the pointwise decay of the vorticity in the far range.}

As before, take a fixed $R_0 > 0$, and consider a smooth $1$-form $v \in \Lambda_{\sigma}^1 (\Omega (R_0)),$ which is a solution to \eqref{NSforvorticity} on the exterior domain $\Omega (R_0) = \mathbb{H}^2(-a^2)-\overline{B_O(R_0)}$. Let $\omega = * \dd v$ be the associated vorticity function of $v$. Then, as before it follows that $\omega$ satisfies  \begin{equation}\label{VorticityeqNEW}
- \Delta \omega + 2a^2 \omega + g(v ,\nabla \omega) = 0.
\end{equation}
Also, estimate \eqref{Vorticity2.23} of Theorem \ref{L2propertyvorticity} informs us that $\omega$ has the following property.
\begin{equation*}
\nabla \omega \in L^2 (\Omega (R_0 + 1 )).
\end{equation*}
Next, we take any $x_0 \in \Omega (R_0 +2)$, so that we have $\rho (x_0) > R_0 +2$, and note 
\begin{equation*}
B_{x_0}(1) \subset \Omega (R_0 + 1 ) .
\end{equation*}
Now, since $\omega \in L^2(\Omega (R_0))$, and $\nabla \omega  \in L^2(\Omega (R_0 + 1 ))$, we can at once deduce that
\begin{equation}\label{limitingW12}
\lim_{\rho (x_0)\rightarrow \infty} \big \| \omega \big \|_{W^{1,2}(B_{x_0}(1))} = 0 .
\end{equation}
Since $\mathbb{H}^2(-a^2)$ is homogeneous in that its spatial structure around one reference point is \emph{identical} to its spatial structure around any other reference point, up to some isometric transformation on $\mathbb{H}^2(-a^2)$, we can simply regard the base point $x_0$ as the vertex $(\frac{1}{a}, 0 , 0)$ of the hyperboloid model of $\mathbb{H}^2(-a^2)$. Under this identification of $x_0$ with the point $(\frac{1}{a} , 0 , 0)$ in the hyperboloid model, the coordinate system
$Y : \mathbb{H}^2(-a^2) \rightarrow D_0(1)$ as defined in subsection \ref{Hyperboloid} now maps $x_0$ to the center $0$ of the unit disc $D_0(1)$. It is equally obvious that $Y$ maps the geodesic ball $B_{x_0}(1)$ onto $D_0(\tanh (\frac{a}{2}))$. 
 
Now, under the following local coordinate system
\begin{equation*}
Y : B_{x_0}(1) \rightarrow D_0(\tanh (\frac{a}{2})),
\end{equation*}
the vorticity equation \eqref{VorticityeqNEW} as restricted on $B_{x_0} (1)$ now will have the following local representation on the Euclidean disc $D_0(\tanh (\frac{a}{2}))$.
\begin{equation}\label{localversionofVorteq}
\Delta^{\mathbb{R}^2} \omega^{\sharp} + \frac{8}{\big ( 1 - |y|^2\big )^2} \omega^{\sharp} + \sum_{1\leq i \leq 2} v_i^{\sharp} \partial_i \omega^{\sharp} = 0 ,
\end{equation}
where $\omega^{\sharp} = \omega \circ Y^{-1}$ and $v_j^{\sharp} = v_j \circ Y^{-1}$ (Recall that $v = v_1 \dd Y^1 + v_2 \dd Y^2$).
Also, a computation shows
\begin{equation}\label{boringONE}
\frac{1}{C_a} \big \|\omega \big \|_{H^1 (B_{x_0}(1))} \leq \big \| \omega^{\sharp} \big \|_{H^1(D_0(\tanh (\frac{a}{2})))} \leq C_a \big \| \omega \big \|_{H^1(B_{x_0}(1))}
\end{equation}
for some absolute constant $C_a > 1$, which depends only on $a$.

Now, we can apply standard local elliptic regularity \cite[Section 8.9, Thm 8.24]{GilbargTrudinger} directly to $\omega^{\sharp}$, as a solution to equation \eqref{localversionofVorteq}, to deduce that there exists a  constant $\widetilde{C}(a, \|v\|_{\infty}) >0$ such that $\omega^{\sharp}$ satisfies the following apriori estimate.

\begin{equation}\label{boringTWO}
\big \| \omega^{\sharp} \big \|_{L^{\infty}(D_0(\frac{1}{2} \tanh(\frac{a}{2})))} \leq \widetilde{C}(a, \|v\|_{\infty}) \big \| \omega^{\sharp} \big \|_{L^2(D_0(\tanh(\frac{a}{2})))}.
\end{equation}

Recall that
\begin{equation}
r(a) = \frac{1}{a} \log \Big ( \frac{1+ 3e^a}{3+ e^a} \Big ) ,
\end{equation}
which satisfies the property that the coordinate chart $Y$ maps the geodesic disc $B_{x_0} (r(a))$ diffeomorphically onto $D_0(\frac{1}{2}\tanh (\frac{a}{2}))$.
So, through combining \eqref{boringONE} with \eqref{boringTWO}, we yield the following estimate.
\begin{equation}
\begin{split}
\big \| \omega \big \|_{L^{\infty} (B_{x_0}(r(a))) } & = \big \| \omega^{\sharp} \big \|_{L^{\infty} (D_0(\frac{1}{2}\tanh (\frac{a}{2})))} \\
& \leq \widetilde{C}(a, \|v\|_{\infty})  \big \| \omega^{\sharp} \big \|_{L^2(D_0(\tanh(\frac{a}{2})))}.\\
& \leq \widetilde{C}(a, \|v\|_{\infty})  C_a \big \| \omega \big \|_{H^1 (B_{x_0}(1))}.
\end{split}
\end{equation}
So, the limiting property \eqref{limitingW12} together with the above estimate gives
\begin{equation}\label{supdecay}
\lim_{\rho (x_0) \rightarrow \infty } \big \| \omega \big \|_{L^{\infty} (B_{x_0}(r(a))) }  = 0,
\end{equation}
which confirms the fact that $\omega (x) \rightarrow 0$, as $\rho (x) \rightarrow \infty$.\\

Armed with \eqref{supdecay} we are finally ready to deduce the exponential decay rate for $\omega (x)$, as $\rho (x) \rightarrow \infty$.

\subsection{Proof of Theorem \ref{FinalThmforExpdecay}}
\begin{proof}
To begin, using the identity $\Delta \rho = a \coth (a \rho)$, and $g(\nabla \rho, \nabla \rho)=1$, we compute
\begin{equation}
\begin{split}
\Delta \big (  e^{-\delta \rho } \big ) & =\dv (\nabla  e^{-\delta \rho } )\\
&=\dv (-\delta e^{-\delta \rho} \nabla \rho )\\
&=-\delta g(\nabla  e^{-\delta\rho},\nabla \rho)-\delta e^{-\delta \rho}\Delta \rho\\
&= \delta^2   e^{-\delta \rho} - \delta   a   e^{-\delta \rho}  \coth (a \rho ) \\
& \leq e^{-\delta \rho}  \big \{ \delta^2 - \delta   a \big \}.
\end{split}
\end{equation}
Next, smoothness of $v$, and Theorem \ref{VelocityDecayThm} imply
\begin{equation*}
\big \| v \big \|_{L^{\infty}(\Omega (R_1) )} < \infty.
\end{equation*}
In what follows, we use the abbreviation $\|v\|_{\infty}$ for $\big \| v \|_{L^{\infty}(\Omega(R_1) )}$. Observe the following straightforward estimate holds pointwise on $\Omega (R_1)$.
\begin{equation*}
\begin{split}
\big |g( v,   \nabla e^{-\delta \rho } )\big |  & = \big | \delta  e^{-\delta \rho}g( v,   \nabla \rho)   \big | \\
& \leq \delta   \big \| v \big \|_{\infty}   e^{-\delta \rho} .
\end{split} 
\end{equation*} 
Hence
\begin{equation}\label{quiteboringest1}
\begin{split}
L \big ( e^{-\delta \rho } \big ) \leq e^{-\delta \rho} \cdot \Big \{ \delta^2 + \delta \big ( \|v\|_{\infty} - a  \big ) -2 a^2    \Big \} ,
\end{split}
\end{equation}
holds pointwise on $\Omega (R_1)$, where $L$ is the elliptic operator specified in \eqref{eoperator}.

Note that the two distinct roots of the quadratic equation $t^2 + \big( \|v\|_{\infty} - a \big )t -2a^2 = 0$ are given by
\begin{equation*}
\begin{split}
\tau_1 & = -\frac{1}{2}\Big \{ \sqrt{( \|v\|_{\infty} - a )^2 + 8a^2     } + (\|v\|_{\infty} - a)  \Big \} ,\\
\tau_2 & = \frac{1}{2} \Big \{  \sqrt{ ( \|v\|_{\infty} - a )^2 + 8a^2   } - (\|v\|_{\infty} - a)  \Big \} .
\end{split}
\end{equation*} 

It is obvious that $\tau_1 < 0 < \tau_2$, and that the relation $t^2 + \big( \|v\|_{\infty} - a \big )t -2a^2 < 0$ holds for any $t \in (\tau_1 , \tau_2)$. So, by just taking $t$ to be the constant $\delta (a, \|v\|_{\infty})$ as specified in \eqref{definedelta}, we yield the following estimate.
\begin{equation}\label{quitetrivialest1}
\big ( \delta (a, \|v\|_{\infty}) \big )^2 + \big( \|v\|_{\infty} - a \big )\cdot \delta (a, \|v\|_{\infty}) - 2a^2< 0 . 
\end{equation}

So,  \eqref{quiteboringest1} and \eqref{quitetrivialest1} give
\begin{equation}\label{beingsupsolution}
L \big ( e^{-\delta \rho } \big ) < 0 .
\end{equation}
Consider now the positive constant $A$ which is specified in \eqref{definitionofA}.  Then the functions $Ae^{-\delta \rho}$ and $-Ae^{-\delta \rho}$ are supersolution and subsolution of $L$, respectively.

Next, by definition of $A$,  the desired estimate \eqref{Finalexpdecay} holds so far for any 
$x \in \partial B_O(R_1)$, so 
\begin{equation}\label{boundaryproperty}
\begin{split}
 -A e^{-\delta (R_1)} \leq \omega \big |_{\partial B_O(R_1)} \leq -A e^{-\delta (R_1)} .
\end{split}
\end{equation}   
 We note that we would like \eqref{Finalexpdecay} to hold for all $x\in \Omega(R_1)$.  To see that this is in fact the case, we recall we have
 \[
  \lim_{|x| \rightarrow \infty } \omega (x) = 0 
 \]
so this allow us to use the comparison principle for the operator $L$ to deduce that estimate \eqref{Finalexpdecay} does hold for any $x \in \Omega (R_1)$. This completes the proof of Theorem \ref{FinalThmforExpdecay}.
\end{proof}

\section{Pressure: proof of Theorem \ref{thm_p}}\label{section_p}
Due to the work of Anderson \cite{Anderson} and Sullivan \cite{Sullivan}, we know there exists a smooth and bounded harmonic function $F$ that comes from a continuous boundary data $\phi$ at infinity (see also \cite{AndersonSchoen}).  If the boundary data is chosen to be non-constant, then $F$ is nontrivial.  Now let $v=\dd F$,  and $P=-2a^2F-\frac 12\abs{\dd F}_a^2$, then \eqref{StatNSforvelocityDecay} is satisfied since (more details for these computations can be found in \cite{CC10})
\begin{align*}
2\Def^*\Def +\nabla_v v&=-\Delta v-2 \Ric(v)+ \nabla_v v\\
&=-\Delta (\dd F)-2\Ric(\dd F)+\frac 12\dd \abs{\dd F}_a^2\\
&=2a^2\dd F+\frac 12\dd \abs{\dd F}_a^2,
\end{align*}
and as shown in \cite{CC10}, $\abs{\dd F}_a \rightarrow 0$ at infinity, so $P\rightarrow-2a^2F=-2a^2 \phi\neq $ constant as needed.
\appendix

 \section{Standard sup norm estimates}\label{appendixa}

The following is a derivation of what should be a standard $L^\infty$ estimate for the solution of the Stokes equation, and we only include it here for completeness.  It is based on \cite{Seregin}, and we write it in the form that we apply it in the paper.
 
For each $r > 0$, we consider the Eucldiean disc $D_0(r) = \{ x \in \mathbb{R}^2 : |x| < r\}$. Consider a vector valued function $u \in C^{\infty} (D_0(1))$, and a function $P \in C^{\infty}(D_0(1))$ which satisfies the Stokes equation
\begin{equation}\label{EASYLINEARSTOKESNEW}
\begin{split}
-\Delta^{\mathbb{R}^2} u + \nabla^{\mathbb{R}^2} P & = F ,\\
\dv u & = 0 ,
\end{split}
\end{equation}
where the external force  $F \in C^{\infty}(D_0(1)) \cap L^{\frac{4}{3}} (D_0(1))$. Our goal here is to derive an a priori estimate for
 \begin{equation*}
\big \|  \big ( \nabla^{\mathbb{R}^2} \big )^2 u   \big \|_{L^{\frac{4}{3}} (D_0(\frac{1}{2}))} 
\end{equation*}
  in terms of $\|F\|_{L^{\frac{4}{3}} (D_0(1))}$, $\|u\|_{L^2(D_0(1))}$ and $\|\nabla^{\mathbb{R}^2} u\|_{L^2(D_0(1))}$.

To this end, we first carry out the following estimate, which holds for any test vector field $\varphi \in C^{\infty}_{c} ( D_0(1))$.
\begin{equation*}
\begin{split}
\bigg | \int_{D_0(1)} F \cdot \varphi \bigg | & \leq \big \|F\big \|_{L^{\frac{4}{3}} (D_0 (1))} \big \| \varphi \big \|_{L^4 (D_0(1))} \\
& \leq \big \|F\big \|_{L^{\frac{4}{3}} (D_0 (1))}  C_0 \big \| \varphi \big \|_{L^2 (D_0(1))}^{\frac{1}{2}}  \big \| \nabla^{\mathbb{R}^2} \varphi \big \|_{L^2 (D_0(1))}^{\frac{1}{2}} \\
& \leq C_0\big \|F\big \|_{L^{\frac{4}{3}} (D_0 (1))}\big \| \nabla^{\mathbb{R}^2} \varphi \big \|_{L^2 (D_0(1))} ,
\end{split}
\end{equation*}
which gives
\begin{equation}\label{EASYFNEW}
\big \| F \big \|_{L^{-1,2} (D_0(1))} \leq C_0 \|F\big \|_{L^{\frac{4}{3}} (D_0 (1))}.
\end{equation}
It is plain to see that we have
\begin{equation}\label{TRIVIALNEW}
\big \| -\Delta^{\mathbb{R}^2} u \big \|_{L^{-1,2}(D_0(1))} \leq \big \| \nabla^{\mathbb{R}^2} u \big \|_{L^2(D_0(1))} .
\end{equation}
So \eqref{EASYLINEARSTOKESNEW}, \eqref{EASYFNEW} and \eqref{TRIVIALNEW} give
\begin{equation}\label{EASYnablaPressureNEW}
\big \| \nabla^{\mathbb{R}^2} P  \big \|_{L^{-1,2} (D_0 (1))} \leq C_{0} \Big(\big \|F\big \|_{L^{\frac{4}{3}} (D_0 (1))} + \big \| \nabla^{\mathbb{R}^2} u \big \|_{L^2(D_0(1))}\Big). \end{equation}
The fact that $\nabla^{\mathbb{R}^2} P \in L^{-1,2} (D_0(1))$ implies that there exists some $c \in \mathbb{R}$ for which the following a priori estimate holds \cite{Seregin}.
\begin{equation}\label{EASYPressureNEW}
\big \| P - c \big \|_{L^2 (D_0(1))} \leq C_0\big \| \nabla^{\mathbb{R}^2} P  \big \|_{L^{-1,2} (D_0 (1))} .
\end{equation}
Hence \eqref{EASYnablaPressureNEW} and \eqref{EASYPressureNEW} together give 
\begin{equation}\label{pressureLTWONEW}
\big \| P - c \big \|_{L^2 (D_0(1))} \leq C_{0} \Big(\big \|F\big \|_{L^{\frac{4}{3}} (D_0 (1))} + \big \| \nabla^{\mathbb{R}^2} u \big \|_{L^2(D_0(1))} \Big).
\end{equation}
We can now rephrase the Stokes equation \eqref{EASYLINEARSTOKESNEW} as follows, with $P$ replaced by $P-c$.
\begin{equation}\label{EASYLINEARSTOKESNEWV2}
\begin{split}
-\Delta^{\mathbb{R}^2} u + \nabla^{\mathbb{R}^2} \big ( P-c \big ) & = F ,\\
\dv u & = 0 .
\end{split}
\end{equation}
Next, to localize $u$ to the ball $D_0(\frac{1}{2})$, we take a radially symmetric bump function $\eta \in C^{\infty} (D_0(1))$ which satisfies
$\chi_{_{D_0(\frac{1}{2})}} \leq \eta \leq \chi_{_{D_0(1)}}$. Then, it follows that $w = \eta u\in C^{\infty}(D_0(1))$ satisfies the following system of equations.
\begin{equation}\label{PDEforCoerciveEstNEW}
\begin{split}
-\Delta^{\mathbb{R}^2} w + \nabla^{\mathbb{R}^2} \big\{(P-c)  \eta \big \} & = \eta F -2 \nabla^{\mathbb{R}^2} u \cdot \nabla^{\mathbb{R}^2}\eta - u \Delta^{\mathbb{R}^2} \eta + (P-c)\nabla^{\mathbb{R}^2} \eta ,\\
\dv w & = u\cdot \nabla^{\mathbb{R}^2}\eta .
\end{split}
\end{equation}
By applying the Cattabriga-Solonnikov estimate \cite{Seregin} to system \eqref{PDEforCoerciveEstNEW}, we deduce that $w$ satisfies
\begin{align}
& \big \| \big ( \nabla^{\mathbb{R}^2}\big )^2 w \big \|_{L^{\frac{4}{3}}(D_0(1))} + \big \| \nabla^{\mathbb{R}^2} \big\{(P-c)  \eta \big \}\big \|_{L^{\frac{4}{3}} (D_0(1))} \nonumber\\
\leq & C_0 \big \| \eta F -2 \nabla^{\mathbb{R}^2} u \cdot \nabla^{\mathbb{R}^2}\eta - u \Delta^{\mathbb{R}^2} \eta + (P-c)\nabla^{\mathbb{R}^2} \eta \big \|_{L^{\frac{4}{3}} (D_0(1)) } + C_0 \big \| \nabla^{\mathbb{R}^2} \big ( u \cdot \nabla^{\mathbb{R}^2}\eta\big ) \big \|_{L^{\frac{4}{3}} (D_0(1))} \nonumber\\
\leq &C_0 \Big  \{ \big \|F \big \|_{L^{\frac{4}{3}} (D_0(1))} + \big \| \nabla^{\mathbb{R}^2} \eta \big \|_{L^{\infty} (D_0(1))} \big \| \nabla^{\mathbb{R}^2} u \big \|_{L^{\frac{4}{3}} (D_0(1))}  + \big \| \big ( \nabla^{\mathbb{R}^2} \big )^2  \eta \big \|_{L^{\infty} (D_0(1))} \big \| u \big \|_{L^{\frac{4}{3}} (D_0(1))} \nonumber\\
& \qquad+   \big \| \nabla^{\mathbb{R}^2} \eta \big \|_{L^{\infty} (D_0(1))} \big \| P- c \big \|_{L^{\frac{4}{3}}(D_0(1))} \Big \} \nonumber\\
\leq & C_0 \Big  \{  \big \|F \big \|_{L^{\frac{4}{3}} (D_0(1))} + \big \| u \big \|_{L^2(D_0(1))}  + \big \| \nabla^{\mathbb{R}^2} u \big \|_{L^2(D_0(1))}
+ \big \| P - c \big \|_{L^2(D_0(1))}  \Big  \}, \label{CoerciveL4over3NEWEASY}
\end{align}
where in the last line we use the Holder's estimate $\|f\|_{L^{\frac{4}{3}}(D_0(1))} \leq \big |D_0(1) \big |^{\frac{1}{4}}  \|f\|_{L^2(D_0(1))}$. Now, \eqref{pressureLTWONEW} and \eqref{CoerciveL4over3NEWEASY} together with the fact that $\eta\equiv 1$ on $D_0(\frac 12)$ we have
\begin{equation*}
\begin{split}
\big \| \big ( \nabla^{\mathbb{R}^2}\big )^2 u \big \|_{L^{\frac{4}{3}}(D_0(\frac{1}{2}))}
& \leq \big \| \big ( \nabla^{\mathbb{R}^2}\big )^2 w \big \|_{L^{\frac{4}{3}}(D_0(1))} \\
& \leq C_0 \Big  \{ \big \|F \big \|_{L^{\frac{4}{3}} (D_0(1))} + \big \| u \big \|_{L^2(D_0(1))}  + \big \| \nabla^{\mathbb{R}^2} u \big \|_{L^2(D_0(1))} \Big  \} .
\end{split}
\end{equation*}
By the standard Sobolev embedding, the above estimate gives
\begin{equation*}
\begin{split}
\big \| \nabla^{\mathbb{R}^2} u  \big \|_{L^4(D_0(\frac{1}{2}))} & \leq C_0 \Big  \{ \big \| \nabla^{\mathbb{R}^2} u  \big \|_{L^{\frac{4}{3}}(D_0(\frac{1}{2}))} + \big \| \big ( \nabla^{\mathbb{R}^2} \big )^2 u  \big \|_{L^{\frac{4}{3}}(D_0(\frac{1}{2}))}\Big  \} \\
& \leq C_0 \Big \{ \big \|F \big \|_{L^{\frac{4}{3}} (D_0(1))} + \big \| u \big \|_{L^2(D_0(1))}  + \big \| \nabla^{\mathbb{R}^2} u \big \|_{L^2(D_0(1))} \Big  \}.
\end{split}
\end{equation*}
Of course, the standard Sobolev embedding also gives
\begin{equation*}
\begin{split}
\big \|  u  \big \|_{L^4(D_0(\frac{1}{2}))} & \leq C_0 \Big \{ \big \|  u  \big \|_{L^{\frac{4}{3}}(D_0(\frac{1}{2}))} + \big \| \nabla^{\mathbb{R}^2} u  \big \|_{L^{\frac{4}{3}}(D_0(\frac{1}{2}))}\Big  \} \\
& \leq C_0 \Big \{ \big \|  u  \big \|_{L^{2}(D_0(\frac{1}{2}))} + \big \| \nabla^{\mathbb{R}^2} u  \big \|_{L^{2}(D_0(\frac{1}{2}))} \Big  \}. \end{split}
\end{equation*}
Since we have the Morrey's type embedding $W^{1,4} (D_0(\frac{1}{2})) \subset C^{0,\frac{1}{2}} (D_0(\frac{1}{2}))$, it follows from the above two estimates that
\begin{equation}\label{yetanothersimplebutusefulestimate}
\begin{split}
\big \| u \big \|_{C^{0,\frac{1}{2}} (D_0(\frac{1}{2}))} & \leq C_0 \Big  \{ \big \|  u  \big \|_{L^4(D_0(\frac{1}{2}))} +  \big \| \nabla^{\mathbb{R}^2} u  \big \|_{L^4(D_0(\frac{1}{2}))}     \Big  \} \\
& \leq C_0 \Big  \{ \big \|F \big \|_{L^{\frac{4}{3}} (D_0(1))} + \big \| u \big \|_{L^2(D_0(1))}  + \big \| \nabla^{\mathbb{R}^2} u \big \|_{L^2(D_0(1))}\Big \} .
\end{split}
\end{equation}
In the above argument, we have already established the following useful lemma.

\begin{lemma}\label{LinfityregularityNEW}
Consider a vector field $u \in C^{\infty} (D_0(1)) \cap W^{1,2} (D_0(1))$, and a function $P \in C^{\infty} (D_0(1))$ which together satisfy the following Stokes equation, with the external force $F \in C^{\infty}(D_0(1)) \cap L^{\frac{4}{3}} (D_0(1))$.
\begin{equation}\label{EASYLINEARSTOKESNEWTWO}
\begin{split}
-\Delta^{\mathbb{R}^2} u + \nabla^{\mathbb{R}^2} P & = F ,\\
\dv u & = 0 .
\end{split}
\end{equation}
Then, it follows that $u$ satisfies the following a priori estimate, with $C_0 > 0$ to be some absolute constant which depends only on the dimension of $\mathbb{R}^2$.
\begin{equation}\label{BootstrapNEW}
\big \| u \big \|_{L^{\infty} (D_0(\frac{1}{2}))} \leq C_0\Big  \{ \big \|F \big \|_{L^{\frac{4}{3}} (D_0(1))} + \big \| u \big \|_{L^2(D_0(1))}  + \big \| \nabla^{\mathbb{R}^2} u \big \|_{L^2(D_0(1))} \Big  \} .
\end{equation}
\end{lemma}
\begin{remark}
Indeed, in the estimate \eqref{yetanothersimplebutusefulestimate}, $ \big \| u \big \|_{C^{0,\frac{1}{2}} (D_0(\frac{1}{2}))}$ can be replaced by  $\big \| u \big \|_{L^{\infty} (D_0(\frac{1}{2}))}$.  We decide to drop the Holder's semi-norm, since this will help us get a cleaner estimate in the process of rescaling a solution to \eqref{EASYLINEARSTOKESNEWTWO}.
\end{remark}
Now, fix $R > 0$. Suppose that we have a vector field $u \in C^{\infty} (D_0(R)) \cap H^1 (D_0(R))$ and a function $P \in C^{\infty} (D_0(R))$, which satisfy the linear Stokes equation \eqref{EASYLINEARSTOKESNEW} on $D_0(R)$, with an external force $F \in C^{\infty} (D_0(R))\cap  L^{\frac{4}{3}} (D_0(R))$. Now, we consider the rescaled functions
$u_{_{R}} : D_0(1) \rightarrow \mathbb{R}^2$ and $P_{_{R}} : D_0(1) \rightarrow \mathbb{R}$ defined by
\begin{equation*}
\begin{split}
u_{_{R}}(y) & = R^{-2} u(R y) ,\\
P_{_{R}}(y) & = R^{-1} P (Ry) .
\end{split}
\end{equation*}
Then, the pair $(u_{_{R}} , P_{_{R}} )$ is a solution to the following linear Stokes equation on $D_0(1)$.
\begin{equation*}
\begin{split}
-\Delta^{\mathbb{R}^2} u_{_{R}} + \nabla^{\mathbb{R}^2} P_{_{R}} & = F_{_{R}} ,\\
\dv u_{_{R}} & = 0 ,
\end{split}
\end{equation*}
where $F_{_{R}} : D_0(1) \rightarrow \mathbb{R}^2$ is given by $F_{_{R}}(y) = F (R\cdot y )$, for all $y \in D_0(1)$. By applying estimate \eqref{BootstrapNEW} in Lemma \ref{LinfityregularityNEW} directly to the pair $(u_{_{R}} , P_{_{R}})$, we yield the following estimate
\begin{equation}\label{BootstrapRescaled}
\big \| u_{_{R}} \big \|_{L^{\infty} (D_0(\frac{1}{2}))} \leq C_0 \Big  \{ \big \|F_{_{R}} \big \|_{L^{\frac{4}{3}} (D_0(1))} + \big \| u_{_{R}} \big \|_{L^2(D_0(1))}  + \big \| \nabla^{\mathbb{R}^2} u_{_{R}} \big \|_{L^2(D_0(1))} \Big  \} .
\end{equation}
Observe that we have
\begin{equation*}
\begin{split}
\big \|F_{_{R}} \big \|_{L^{\frac{4}{3}} (D_0(1))} & = R^{-\frac{3}{2}}  \big \|F \big \|_{L^{\frac{4}{3}}(D_0(R))}, \\
\big \| u_{_{R}} \big \|_{L^2(D_0(1))} & = R^{-3} \big \| u \big \|_{L^2(D_0(R))}, \\
\big \| \nabla^{\mathbb{R}^2} u_{_{R}}\big \|_{L^2(D_0(1))} & = R^{-2} \big \| \nabla^{\mathbb{R}^2} u \big \|_{L^2(D_0(R))}, \\
\big \| u_{_{R}} \big \|_{L^{\infty} (D_0(\frac{1}{2}))} & = R^{-2} \big \| u \big \|_{L^{\infty} (D_0(\frac{R}{2}))} .
\end{split}
\end{equation*}
In light of the above scaling properties, we can rephrase \eqref{BootstrapRescaled} as follows.
\begin{equation*}
\big \| u \big \|_{L^{\infty} (D_0(\frac{R}{2}))} \leq C_0 \Big  \{ R^{\frac{1}{2}} \big \|F \big \|_{L^{\frac{4}{3}} (D_0(R))} + R^{-1}\big \| u \big \|_{L^2(D_0(R))}  + \big \| \nabla^{\mathbb{R}^2} u \big \|_{L^2(D_0(R))} \Big\} .
\end{equation*}
The above argument clearly gives the following rescaled version of Lemma \ref{LinfityregularityNEW}.

\begin{lemma}\label{LinfityrescaledNEW}
Consider a vector field $u \in C^{\infty} (D_0(R)) \cap W^{1,2} (D_0(R))$, and a function $P \in C^{\infty} (D_0(R))$ which together satisfy the following Stokes equation, with the external force $F \in C^{\infty}(D_0(R)) \cap L^{\frac{4}{3}} (D_0(R))$.
\begin{equation}\label{EASYLINEARSTOKESNEWThree}
\begin{split}
-\Delta^{\mathbb{R}^2} u + \nabla^{\mathbb{R}^2} P & = F ,\\
\dv u & = 0.
\end{split}
\end{equation}
Then, it follows that $u$ satisfies the following a priori estimate, with $C_0 > 0$ to be some absolute constant which depends only on the dimension of $\mathbb{R}^2$.
\begin{equation}\label{BootstrapFinalNEW}
\big \| u \big \|_{L^{\infty} (D_0(\frac{R}{2}))} \leq C_0 \Big  \{ R^{\frac{1}{2}}\big \|F \big \|_{L^{\frac{4}{3}} (D_0(R))} + R^{-1} \big \| u \big \|_{L^2(D_0(R))}  + \big \| \nabla^{\mathbb{R}^2} u \big \|_{L^2(D_0(R))} \Big  \} .
\end{equation}
\end{lemma}

\bibliography{ref}
\bibliographystyle{plain}

\end{document}